\documentclass[11pt]{article}

\usepackage[margin=1in]{geometry}
\usepackage{amssymb, color}
\usepackage{hyperref}
\usepackage{indentfirst}
\usepackage{tikz}
\usepackage{caption}
\usepackage[T1]{fontenc}
\usepackage{slashed}

 \usepackage{pgfplots}
\usepackage{graphicx,xcolor}
\usepackage{amsthm,amsfonts,euscript,amsmath,comment,slashed,here,todonotes}
\usepackage[affil-it]{authblk}
\usepackage{mathrsfs}
\usepackage{stmaryrd}
\usepackage[cp1251]{inputenc}

\newtheorem{theorem}{Theorem}[section]
\newtheorem{lemma}[theorem]{Lemma}

\newtheorem{proposition}{Proposition}[section]

\theoremstyle{definition}
\newtheorem{definition}[theorem]{Definition}
\newtheorem{example}[theorem]{Example}

\theoremstyle{remark}

\numberwithin{equation}{section}

\newcommand{\ud}{\,\mathrm{d}}
\newcommand{\p}{\ensuremath{\partial}}
\newcommand{\n}{\ensuremath{\nonumber}}
\newcommand{\eps}{\ensuremath{\varepsilon}}

\newcommand\be{\begin{equation}}
\newcommand\ee{\end{equation}}
\newcommand\bea{\begin{eqnarray}}
\newcommand\eea{\end{eqnarray}}
\newcommand\bi{\begin{itemize}}
\newcommand\ei{\end{itemize}}
\newcommand\ben{\begin{enumerate}}
\newcommand\bena{\begin{enumerate}[(a)]}
\newcommand\een{\end{enumerate}}

\newcommand\bp{\begin{proof}}
\newcommand\ep{\end{proof}}

\allowdisplaybreaks

\title{Higher Regularity Theory for a \\ Mixed-Type Parabolic Equation}
\author{Sameer Iyer \footnote{Department of Mathematics, University of California, Davis, Davis, CA 95616, \url{sameer@math.ucdavis.edu}} \qquad Nader Masmoudi \footnote{NYUAD Research Institute, New York University Abu Dhabi, PO Box 129188, Abu Dhabi, United Arab Emirates. \\ \footnotesize \normalfont
 Courant Institute of Mathematical Sciences, New York University, 251 Mercer Street, New York, NY 10012, USA.  \url{masmoudi@cims.nyu.edu}}}

\usetikzlibrary{arrows,positioning} 
\tikzset{
    >=stealth',
    punkt/.style={
           rectangle,
           rounded corners,
           draw=black, very thick,
           text width=6.5em,
           minimum height=2em,
           text centered},
    pil/.style={
           ->,
           thick,
           shorten <=2pt,
           shorten >=2pt,}
}

\begin{document}

\maketitle

\begin{abstract} In this paper, we study the higher regularity theory of a mixed-type parabolic problem. We extend the recent work of \cite{DMR} to construct solutions that have an arbitrary number of derivatives in Sobolev spaces. To achieve this, we introduce a counting argument based on a quantity called the ``degree".  In the second part of this paper, we apply this existence theory to the Prandtl system near the classical Falkner-Skan self-similar profiles in order to supplement the stability analysis of \cite{IM22} with a rigorous construction argument.
\end{abstract}

\setcounter{tocdepth}{2}
\tableofcontents

\section{Introduction}

\subsection{Setup and Main Results for Burgers}

In this work, we are interested in two problems: the Airy equation and the stationary Prandtl system. The Airy equation reads 
\begin{subequations} \label{modelpb0}
\begin{align}
&\rho \p_t \Omega - \p_\rho^2 \Omega = G, \qquad (t, \rho) \in (0, 1) \times \mathbb{R}\\
&\Omega|_{t = 0} = \Xi_{L}, \qquad \rho > 0 \\
&\Omega|_{t = 1} = \Xi_{R}, \qquad \rho < 0, \\
&\Omega|_{\rho \rightarrow \pm \infty} = 0. 
\end{align}
\end{subequations}
Above, the source term $G$ will be considered an abstract source term. The functions $(\Xi_L(\cdot), \Xi_R(\cdot))$ are the given data. Such a problem is of ``mixed-type": for $\rho > 0$, the equation is forward parabolic ($+t$ is the timelike direction). However, when $\rho < 0$ the equation is backwards parabolic ($-t$ is the timelike direction). The given data reflects this due to the fact that $\Xi_L$ is prescribed at $t = 0, \rho > 0$ to initiate the forward evolution, whereas $\Xi_R$ is prescribed at $t = 1, \rho < 0$ to initiate the backwards evolution.

This model has strong solutions that are $C^1_{\rho}$ across the ``interface", $\{\rho = 0\}$. However, one cannot expect to \textit{generically} upgrade these strong solutions in such a way that $\p_t \Omega$ is also $C^1_{\rho}$ across $\{\rho = 0\}$ (even by prescribing smooth data). This observation was made in the very interesting recent work of \cite{DMR}, which we use as a starting point in the current paper. Indeed, the analysis of \cite{DMR} informally states that one needs to ``project away" from two directions of instability in order to take one $\p_t$: 
\begin{theorem}[Dalibard-Marbach-Rax, \cite{DMR}, Informal Statement] There exists a codimension-two manifold of $(\Xi_L, \Xi_R, G)$ which guarantees that the unique strong solution, $\Omega$ to \eqref{modelpb0} enjoys enhanced $\p_t$ regularity: $\Omega \in H^1((0, 1), H^1(\mathbb{R}))$.
\end{theorem}

Our first contribution is to continue the study of \cite{DMR} for the linear problem \eqref{modelpb0} in two senses: 
\begin{itemize}
\item[1.] We view the source term, $G$, as given. We do not require that $G$ itself satisfy any orthogonality condition. Rather, \textit{given any $G$}, we can adjust the data $(\Xi_L, \Xi_R)$ accordingly. This is a very important feature for nonlinear, nonlocal problems (such as our eventual application to the stationary Prandtl system). The ability to do this corresponds to ensuring \textit{surjectivity} of certain linear functionals (which is of course the same as \textit{nontriviality}). We summarize this in the following: 
\begin{proposition} \label{PropconINTRO} Let $G \in H^1((0, 1); L^2(\mathbb{R}))$ There exist nontrivial, independent linear functionals $\ell_1, \ell_0$, and linear functionals $\beta_0, \beta_1$ such that if the following constraints hold
\begin{subequations} \label{const:const:I}
\begin{align} \label{con:a:I}
\Xi_{L}''(0) = &G_{Left}(0), \\ \label{con:b:I}
\Xi_{R}''(0) = & G_{Right}(0), \\  \label{con:c:I}
\ell_1(\Xi_{L}, \Xi_{R}) = & \beta_1(G) \\ \label{con:d:I}
\ell_0(\Xi_{L}, \Xi_{R}) = & \beta_0(G),
\end{align}
\end{subequations}
then the unique weak solution to \eqref{modelpb0} satisfies $\Omega \in H^1((0, 1); H^1(\mathbb{R}))$.
\end{proposition}

\item[2.] We upgrade the number of $\p_t$ derivatives to an arbitrary number $k_\ast$ by putting $2$ extra constraints for each $\p_t$ (and therefore $2k_\ast$ in total). While this process may ``seem natural", it turns out to require significantly new ideas to go from $\p_t$ to $\p_t^2$, etc... Essentially, this is due to the fact that the linear functionals which define the constraints may start to become linearly dependent in principle. We therefore need some arguments inspired by control theory and some delicate counting arguments (based on a newly introduced notion of ``degree") to rule out possible linear dependencies.  
\end{itemize}

Indeed, our main result regarding the linear problem \eqref{modelpb0} is as follows. 
\begin{theorem} \label{thm:toy:main}Let $G$ be a given, smooth source term in \eqref{modelpb0}. Fix any $(\bar{\Xi}_L, \bar{\Xi}_R) \in \bar{P}_{k_\ast} \subset C^\infty_c(\mathbb{R}_+) \otimes C^\infty_c(\mathbb{R}_+)$ (which will be introduced below in \eqref{Pbarkast}). Let $\chi: \mathbb{R}_+ \mapsto \mathbb{R}$ be a fixed cutoff function equal to $1$ on $[0, 1)$ and vanishing on $(2, \infty)$. There exist $4k_\ast $numbers (which depend on $G$):
\begin{align}
\bold{A}_{k_\ast} := (c_0^{(k)}, c_1^{(k)}, q_{L, 2 + 3(k-1)}, q_{R, 2 + 3(k-1)}), \qquad 1 \le k \le k_\ast, 
\end{align}
and smooth, compactly supported functions $\bold{e}_0^{(k)}, \bold{e}_1^{(k)}$, such that data of the type 
\begin{subequations} \label{data:set}
\begin{align} 
(\Xi_L, \Xi_R) = & ( \widehat{\Xi}_L, \widehat{\Xi}_R) +(Q_L, Q_R), \\
( \widehat{\Xi}_L, \widehat{\Xi}_R) = &(\bar{\Xi}_L, \bar{\Xi}_R) + \sum_{k = 1}^{k_\ast} c^{(k)}_0 \bold{e}^{(k)}_0 + c^{(k)}_1 \bold{e}^{(k)}_1 \\
(Q_L, Q_R) = & (\chi(\rho) \sum_{k = 1}^{k_\ast} q_{L, 2 + 3(k-1)} \rho^{2 + 3(k-1)}, \chi(\rho) \sum_{k = 1}^{k_\ast} q_{R, 2 + 3(k-1)} \rho^{2 + 3(k-1)} )
\end{align}
\end{subequations}
result in the unique strong solution $\Omega$ to \eqref{modelpb0} having the higher regularity $H^{k_\ast}(0, 1, H^1(\mathbb{R}_+))$. 
\end{theorem}

\subsection{Setup and Main Results for Prandtl}

The stationary Prandtl system is a much more complicated system due to many reasons (it is quasilinear, it is nonlocal due to the incompressibility, it is not even linearly the same as \eqref{modelpb0}, etc...). This system reads:
\begin{align}
\begin{aligned} \label{eq:PR:0}
&u_P \p_x u_P + v_P \p_y u_P - \p_y^2 u_{P} = -\p_x p_E(x), \\
&\p_x u_P + \p_y v_P = 0, \qquad (x,y) \in \Omega_L := (1, 1+L) \times (0, \infty), \\
&u_P|_{y = 0} = v_P|_{y = 0} = 0, \qquad \lim_{y \rightarrow \infty} u_P(x, y) = u_E(x).
\end{aligned}
\end{align}
Above, the function $u_E(x)$ is the Euler flow to which the Prandtl boundary layer is forced to match, and the function $\p_x p_E(x)$ is the Euler pressure at $0$. Both of these functions are considered as given data from the point of view of the Prandtl system, and satisfy the Bernoulli law $u_E \p_x u_E = - \p_x p_E$. 

It is well-known that the equation \eqref{eq:PR:0} is parabolic in $x$ as long as $u_P$ (the coefficient in front of the $\p_x$ transport) remains nonnegative. This is due to the formal scaling $u_P \p_x = \p_y^2$. Due to this, it is classical to pose initial data at $x = 0$ for \eqref{eq:PR:0}. We refer the reader to the works \cite{Oleinik}, \cite{GI1}, \cite{IyerBlasius} for rigorous analysis of the Prandtl equations in this regime.  

When the Euler pressure satisfies the sign condition $\p_x p_E(x) < 0$, we say there is an ``adverse pressure gradient". In this event, the solutions constructed by Oleinik in \cite{Oleinik} are not global solutions. This is due to the famous phenomenon of boundary layer separation. Physically, this corresponds to the coherent boundary layer detaching from the wall (say at $\{y = 0\}$) and ejecting into the fluid domain. As far as rigorous analysis of boundary layer separation for the Prandtl equations, we point the reader to the works \cite{MD} and \cite{Zhangsep}.

In this work, and the companion work \cite{IM22}, we are interested in \textit{reversed} flows, which occur after boundary layer separation. This regime can be visualized by consulting Figure \ref{Fig:1} below. Physically, this regime is distinguished from before separation because $u_P$ changes sign: $u_P < 0$ below the dotted line in Figure \ref{Fig:1} (hence the terminology ``reversed" flow). 
\begin{figure}[h] 
\hspace{28 mm} \includegraphics[scale=0.2]{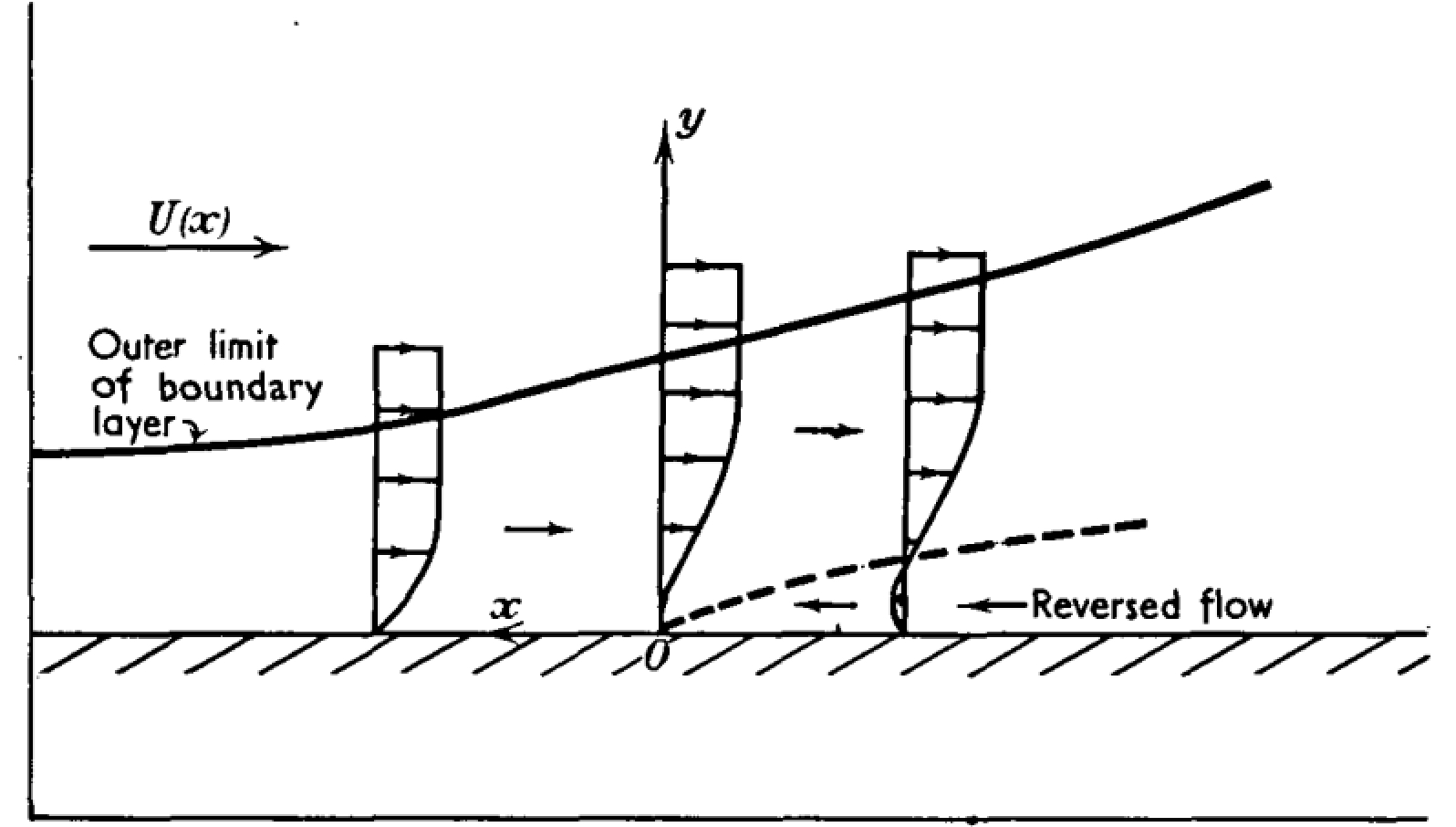}
\caption{Onset and continuation past separation, \cite{Stew0}} \label{Fig:1}
\end{figure}
Mathematically, the sign change of $u_P$ prevents us from regarding the system \eqref{eq:PR:0} as an evolution in $x$. Instead, we view \eqref{eq:PR:0} as a quasilinear, nonlocal, mixed-type problem. To emphasize this point of view further, we direct the reader's attention to Figure \ref{fig:org:mix}.
\begin{figure}[h]
\centering
\begin{tikzpicture}
\draw[ultra thick, <-] (-.5,3) node[above]{$y$} -- (-.5,-0.5);
\draw[ultra thick, -] (4,1.5) -- (4,-2);
\draw[ultra thick, -] (-.5,-2) -- (4,-2);
\node [below] at (1.75, -2) {$u_P = 0, \psi_P = 0$};
\draw[ultra thick,blue, -] (-.5, -.5) to[bend left] (4,1.5);
\node [below] at (1.75, 1) {$\textcolor{blue}{u_P =0}$};
\node [below] at (1.75, -.2) {$u_P < 0$};
\node [below] at (1.75, -.6) {\scriptsize $u_P \p_x u_P + \textcolor{cyan}{v_P \p_y u_P} - \p_y^2 u_P = F$};
\node [below] at (1.75, -1.2) {(reversed flow)};
\node [below] at (1.75, 3.2) {$u_P > 0$};
\node [below] at (1.75, 2.7) {\scriptsize $u_P \p_x u_P + \textcolor{cyan}{ v_P \p_y u_P} - \p_y^2 u_P = F$};
\node [below] at (1.75, 2.1) {(forward flow)};
\draw[ultra thick, dashed, -] (-.5,-2) -- (-.5,-.5);
\draw[ultra thick, dashed, ->] (4,1.5) -- (4,3);
\draw[ultra thick, dashed, ->] (4,-2) -- (5,-2);
\draw[ultra thick, dashed, ->] (-.5,-2) -- (-1.5,-2);
\node[right] at (5,-2) {$x$};
\end{tikzpicture}
\caption{Prandtl as a Free Boundary Mixed-Type Problem} \label{fig:org:mix}
\end{figure}
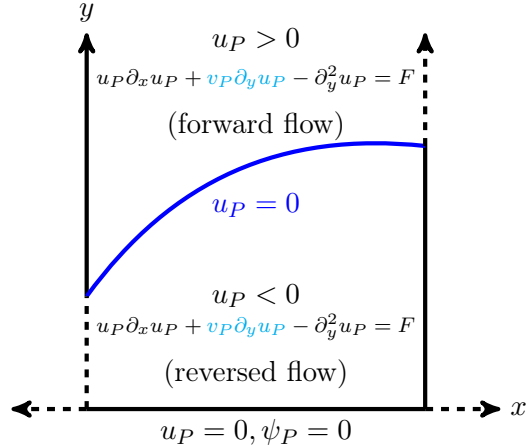

A natural strategy to study reversed flows is to perturb around exact solutions to \eqref{eq:PR:0} that are known to exhibit backflow. Towards this goal, we define
\begin{align} \label{choice:uEx}
u_E(x) := x^n, \qquad n = \frac{\beta}{2-\beta}, \qquad - 0.2 < \beta < 0.
\end{align}
We correspondingly define the Falkner-Skan self-similar profiles:
\begin{align} \label{uFS}
u_{FS}(x, y) := u_E(x) f'(\eta), \qquad \eta := \frac{y}{x^{\frac{1-n}{2}}}. 
\end{align}
The self-similar profile, $f(\cdot)$, satisfies the following equation:
\begin{align} \label{FS:beta}
f''' + f f'' + \beta (1 - (f')^2) = 0, \qquad f(0) = f'(0) = 0, \qquad f'(\infty) = 1,
\end{align}
which is often referred to as a Falkner-Skan self-similar profile to the Prandtl equations. We refer the reader to \cite{Stewartson}, \cite{Stew0}, \cite{Stew64} for further discussion about these self-similar solutions. Our results will all be perturbative around these Falkner-Skan profiles. As a result, it is useful to define the ``remainders":
\begin{align} \label{remsum1}
\psi_P = \psi_{FS} + \eps \psi_{R}, \qquad u_P = u_{FS} + \eps u_R, \qquad v_P = v_{FS} + \eps v_R, \qquad 0 < \eps << 1, 
\end{align}
where $\eps$ is the perturbation amplitude. 

We define 
\begin{align} \n
\mathcal{P}_{general} := \{ (F_{Left}, F_{Right}) | &F_{Left} : \mathbb{R}_+ \mapsto \mathbb{R},  F_{Right}: \mathbb{R}_- \mapsto \mathbb{R}, \\
& \text{supp}(F_{Right}) \subset (-y_{b},0), \| (F_{Left}, F_{Right}) \|_{\mathcal{P}} < \infty \},
\end{align} 
where
\begin{align}
\| (F_{Left}, F_{Right}) \|_{\mathcal{P}_{general}} := \sum_{j = 0}^{J_{Max}} \|\p_{\zeta}^j F_{Left} \langle \zeta \rangle^{3M_{Max}}\|_{L^\infty(\mathbb{R}_+)} + \sum_{j = 0}^{J_{Max}} \|\p_{\zeta}^j F_{Right} \|_{L^\infty(\mathbb{R}_-)}
\end{align}
We leave the parameters $M_{Max}, J_{Max}$ as generic large numbers. $F_{Right}$ is defined on $\mathbb{R}_-$. The parameter $y_b$ will be chosen concretely by $y_b = \frac12 \min\{ 1,  L^{-\frac13}\}$, which guarantees that $\omega_g$ is also compactly supported away from $y = 0$. For ease of exposition, while our proof carries over easily to data in $\mathcal{P}_{general}$, we choose to work with 
\begin{align}
\mathcal{P} = & C^\infty_c(\overline{\mathbb{R}_+}) \times C^\infty_c(\overline{\mathbb{R}_+}).
\end{align} 
We will also introduce some subspaces of $\mathcal{P}$. First, we introduce functions which are compactly supported away from $0$:  
\begin{align}
\widehat{\mathcal{P}} := C^\infty_c(\mathbb{R}_+) \times C^\infty_c(\mathbb{R}_+)
\end{align}
Second, we introduce 
\begin{align} \label{Pbarkast}
\bar{\mathcal{P}}_{k_\ast} := \{ (\widehat{F}_{Left}, \widehat{F}_{Right}) \in \widehat{P} : \ell_0^{(k)} = 0, \ell_1^{(k)} = 0 \text{ for } 1 \le k \le k_\ast \},
\end{align}
where $\ell^{(k)}_0, \ell^{(k)}_1$ will be linear functionals we will introduce later in \eqref{form:ell:0} -- \eqref{form:ell:1}. 

As described in the work \cite{IM22}, we work at the level of the renormalized vorticity, $\omega_g$, due to favorable cancellations in the equation satisfied by $\omega_g$ as opposed to the velocity itself, where $\omega_g$ is defined as follows:
\begin{align}
\omega_{g}(x, y) := &\p_y u_R(x, y) - \frac{\p_y^2 u_{FS}(x, y)}{\p_y u_{FS}(x,y)} u_R(x, y).
\end{align}
We will define the data
\begin{subequations}
\begin{align} \label{dwttm1}
\omega_g|_{x = 1}(y) =&F_{Left}\Big(\frac{1}{L^{\frac13}} (\frac{y}{y_1(\eps)} - 1 )\Big) \qquad y > y_1(\eps), \\ \label{dwttm2}
 \omega_g|_{x = 1+L}(y) = & F_{Right}\Big(\frac{1}{L^{\frac13}} (\frac{y}{y_2} - 1 )\Big), \qquad 0 \le y \le y_2, 
\end{align} 
\end{subequations}
where the number $y_1(\eps)$ is defined by the relation 
\begin{align}
u_{FS}(1, y_1(\eps)) + \eps u_R(1, y_1(\eps)) = 0. 
\end{align}
We will enforce the constraint: 
\begin{align} \label{mc1}
 \int_0^1 \frac{F_{Right}(L^{-\frac13}(z'-1))}{\overline{u_{FS}}'(1, z')} \ud z' = 0. 
\end{align} 

In order to describe our main result, we need to first recall the main result from \cite{IM22}, which are \textit{a-priori} stability estimates: 
\begin{theorem}[Stability, \cite{IM22}] \label{thm:pri:1} Fix any $(F_{Left},  F_{Right}) \in \mathcal{P}$ such that \eqref{mc1} holds, and take $\omega_g$ to achieve the data as shown in \eqref{dwttm1} - \eqref{dwttm2}. Let the maximal tangential length scale $L_{Max} > 0$ be any fixed number. There exist a discrete set of resonant lengths, $\{L_k \}_{k = 1}^N$, $0 < L_k < L_{k+1} < L_{Max}$, depending only on the background profile, $u_{FS}$, and numbers $M_{Max}, J_{Max},$ and $k_{\ast}$, depending only on universal constants such that the following holds. Choose $L \in (0, L_{Max}) \setminus \cup_{k = 1}^N L_k$.  There exists an $\eps_{\ast} > 0$, small relative to $\min_{k} |L - L_k|$ such that for any $0 < \eps < \eps_{\ast}$, we have the \textit{a-priori} stability estimates:
\begin{align} \label{stable:2}
\sum_{k = 0}^{k_{\ast}} \| (L \p_x)^k \psi_R \|_{L^\infty_y L^6_x} + \| (L \p_x)^k u_R \|_{L^\infty_y L^6_x} + \|(L \p_x)^k \omega_R \|_{L^\infty_y L^6_x} \le & o_L(1) \| F_{Left}, F_{Right} \|_{\mathcal{P}} 
\end{align}
for solutions $\omega_g \in H^{k_{\ast}}_x(0, 1;H^{1}_y)$.
\end{theorem}
Above, the pre-factor $o_L(1)$ represents a function, $n(L)$, such that $\lim_{L \downarrow 0} n(L) = 0$ (we do not quantify the behavior of $n(L)$ when $L$ gets large).  The above theorem is an \textit{a-priori} stability estimates, assuming qualitative smoothness on the perturbation $\omega_g$. However, constructing a sufficiently smooth perturbation $u_R$ is a nontrivial task in this scenario. Indeed, for more standard PDE problems, the construction of a smooth enough solution essentially follows from \textit{a-priori} bounds. However, due to the ``mixed-type" nature of the problem here, we need to impose further constraints on the data $(F_{Left}, F_{Right})$ in order to guarantee the existence of regular solutions. 

These constraints were studied in a related setting in the recent work of \cite{DMR}, which studies the Burgers equation:  
\begin{align} \label{Burger}
(y + \eps u) \p_x u - \p_y^2 u = 0, \qquad (x, y) \in (0,1) \times (-1,1). 
\end{align}
The scheme we implement here to construct the solution is similar in spirit to the work of \cite{DMR}. Several complicating factors for the true Prandtl equations make the analysis significantly more involved than the model problem studied in \cite{DMR}. We list these here, and we discuss them in more detail later. 
\begin{itemize}
\item Prandtl contains extra \textit{linear} terms, as compared to the Burgers equation. Indeed, linearizing \eqref{eq:PR:0}, we obtain the linearized system
\begin{align} \label{simpL2}
u_{FS} \p_x u_R + u_R \p_x u_{FS} + v_{FS} \p_y u_R + v_R \p_y u_{FS} - \p_y^2 u_R = 0. 
\end{align}
While one can think of the first transport term $u_{FS}\p_x u - \p_y^2 u_R = 0$ as similar to the linear part of the Burger's equation, \eqref{Burger}, the remaining three transport terms are still linear and create serious difficulties (they cannot easily be moved to the right-hand side, in a way we could potentially do with a nonlinearity). 

\item One of these linear terms, $v_R \p_y u_{FS}$ is \textit{nonlocal}, which can be seen by expressing $v_R = - \int_0^y \p_x u_R$. This is well-known to be a serious difficulty of the Prandtl system. These difficulties are even further amplified in the mixed-type setting. Heuristically, we can think that Figure \ref{fig:org:mix} shows information transported left-to-right above the interface, right-to-left below the interface, but also bottom-to-top due to the $v_R \p_y u_{FS}$ term, all at the linearized level.  

To contend with the difficulties above, we identify new structures and cancellations in the system which allow us to decompose the dynamics into a codimension-two manifold on which ``Airy-dynamics" occur, and a two-dimensional manifold on which we need to use Energy estimates that are more non-local. In turn, closing this argument requires us to excise a finite number of $L$ values. 

\item A third difficulty is due to the introduction of potentially small values of $L$ (which is the main case). It is natural, due to the scaling of the equations, to consider data that are at the ``fast scale" in $L^{\frac13}$ as shown in \eqref{dwttm1} -- \eqref{dwttm2}. However, this fast scale is destroyed at the stream function level through integration. Therefore, even though we impose only fast scales in \eqref{dwttm1} -- \eqref{dwttm2}, the data for higher tangential derivatives immediately inherits slow scales as well. Our analysis relies on careful decompositions by scale: the fast scales typically have a larger amplitude (but highly localized), whereas the slow scales have smaller amplitude (positive powers of $L$) but vary slowly.

\end{itemize}

The main result of this paper is the theorem below. It states that we can find a set of data, $(F_{Left}, F_{Right})$, which guarantees the existence of a sufficiently smooth solution (to which the \textit{a-priori} stability estimates in Theorem \ref{thm:pri:1}  can be applied). 
\begin{theorem}[Existence] \label{thm:twomain} Fix an arbitrary $(\overline{F}_{Left}, \overline{F}_{Right}) \in \overline{\mathcal{P}}_{k_\ast}$. There exists an operator $\mathcal{N}_{side}: \overline{\mathcal{P}}_{k_\ast} \mapsto \mathcal{P}$, a compact perturbation of the identity, such that the following holds. Let $(F_{Left}, F_{Right}) := \mathcal{N}_{side}[\overline{F}_{Left}, \overline{F}_{Right}]$. Then there exists a unique solution $\psi_R$ such that $\psi_P$ constructed from \eqref{remsum1} is a solution to \eqref{eq:PR:0}, and $\omega_g$ achieves the given data as shown in \eqref{dwttm1} -- \eqref{dwttm2}. Moreover, this solution has the regularity $\p_x^k u_R \in L^2_x L^2_y$ for each $k$, $1 \le k \le k_{\ast}$. 
\end{theorem}

Taken together with the stability analysis in Theorem \ref{thm:pri:1}, we have the following: 
\begin{theorem}[Existence and Stability] \label{thmex} Fix an arbitrary $(\overline{F}_{Left}, \overline{F}_{Right}) \in \overline{\mathcal{P}}_{k_\ast}$ where $k_\ast$ is chosen sufficiently large relative to universal constants. There exists an operator $\mathcal{N}_{side}: \overline{\mathcal{P}}_{k_\ast} \mapsto \mathcal{P}$, which is a compact perturbation of the identity, such that the following holds. Let $(F_{Left}, F_{Right}) := \mathcal{N}_{side}[\overline{F}_{Left}, \overline{F}_{Right}]$. Let $L \in (0, L_{Max}) \setminus \cup_{k = 1}^N L_k$. There exists an $\eps_{\ast} > 0$, small relative to $\min_{k} |L - L_k|$ such that for any $0 < \eps < \eps_{\ast}$, the following is valid. 

\begin{itemize}

\item There exists a unique solution $\psi_R$  such that $\omega_g$ achieves the given data as shown in \eqref{dwttm1} -- \eqref{dwttm2}. Moreover, this solution has the regularity $\p_x^k u_R \in L^2_x H^1_y$ for each $k$, $0 \le k \le k_{\ast}$.

\item The solution obeys the stability estimates 
\begin{align}  \label{stable:2:ult}  
\sum_{k = 0}^{k_{\ast}} \| (L \p_x)^k \psi_R \|_{L^\infty_y L^6_x} + \| (L \p_x)^k u_R \|_{L^\infty_y L^6_x} + \|(L \p_x)^k \omega_R \|_{L^\infty_y L^6_x} \le &o_L(1) \| \bar{F}_{Left}, \bar{F}_{Right} \|_{\mathcal{P}}.
\end{align}

\item The data on the sides can be described precisely as follows 
\begin{subequations}
\begin{align}
(F_L, F_R) =& (\widehat{F}_L, \widehat{F}_R) + (Q_L, Q_R) \\ \label{thmdata1}
(\widehat{F}_L, \widehat{F}_R) =& ( \bar{F}_{Left}, \bar{F}_{Right}) +c_{-1}\bold{e}_{-1} + \sum_{k = 1}^{k_\ast} (c_0^{(k)} \bold{e}^{(k)}_0 + c^{(k)}_1 \bold{e}^{(k)}_1) \\ \label{thmdata2}
(Q_L, Q_R) = & (\chi(\rho) \sum_{k = 1}^{k_\ast} q_{L, 2 + 3(k-1)} \rho^{2 + 3(k-1)} , \chi(\rho) \sum_{k = 1}^{k_\ast} q_{R, 2 + 3(k-1)} \rho^{2 + 3(k-1)}  )
\end{align}
\end{subequations}
where $\chi, \bold{e}_{-1}, \bold{e}_0^{(k)}, \bold{e}_1^{(k)}$ are fixed functions, and the $4k_\ast + 1$ numbers
\begin{align}
\bold{A}_{k_\ast} := (c_0^{(k)}, c_1^{(k)}, c_{-1}, q_{L, 2 + 3(k-1)}, q_{R, 2 + 3(k-1)})_{k = 1}^{k_\ast} 
\end{align}
are determined by the prescribed functions $(\bar{F}_{Left}, \bar{F}_{Right}) \in \bar{\mathcal{P}}_{k_\ast}$ and obeys the estimate 
\begin{align}
|\bold{A}_{k_\ast} |_{\mathbb{R}^{4k_\ast + 1}} \lesssim o_L(1) \| (\bar{F}_{Left}, \bar{F}_{Right}) \|_{\bar{\mathcal{P}}_{k_\ast}}.
\end{align}
\end{itemize}

\end{theorem}

\subsection{Related Literature}

There is a wide literature on the Prandtl theory which reflects the many available points of view and settings in which the boundary layer theory has been studied. We will restrict here to the works on the stationary Prandtl system, and we refer the interested reader to our paper \cite{IM22} which contains a much more in-depth discussion of the literature.   

The first rigorous works on the stationary Prandtl system were performed by Oleinik, \cite{Oleinik}, \cite{Oleinik1}. Oleinik introduced the point-of-view that $x$ behaves ``time-like" and parametrizes the evolution of \eqref{eq:PR:0}, (as long as $u_P \ge 0$). Her results proved local well-posedness $0 < L << 1$ of \eqref{eq:PR:0} under very mild hypothesis on the initial data, and global well-posedness under the assumption of a favorable pressure gradient, $\p_x p_E(x) \ge 0$. These results were based upon using maximum principle and comparison principle methods.  

Oleinik's local well-posedness result proves essentially $C^2$ regularity of the solution up to the boundary. Higher regularity was an open question until recently. The work of Guo-Iyer in \cite{GI2} obtained higher-regularity estimates using an energy based method. In addition, the work of Wang-Zhang, \cite{Zhifei:smooth} performed a maximum-principle based analysis to obtain $C^\infty$ regularity up to the boundary. 

Regarding global solutions, which only generically exist in the presence of a favorable pressure gradient, the work of Serrin, \cite{Serrin}, distinguishes the Falkner-Skan profiles as large-$x$ ``attractors" of the dynamics. The work of Iyer, \cite{IyerBlasius} quantifies more precisely the large-$x$ asymptotic stability of specifically the Blasius profile for small, localized initial perturbations. 

When the assumption of a favorable pressure gradient is removed, the physical phenomena of boundary-layer separation discussed above is expected to generically appear. This remarkably was predicted on physical grounds by Prandtl himself in his original work, \cite{Prandtl}. This ``boundary layer separation" was proven rigorously by Dalibard-Masmoudi in \cite{MD} and, using completely different methods, by Shen, Wang, and Zhang, \cite{Zhangsep}. 

We now move to the question of reversed flows. Here, we have our companion work, \cite{IM22}, which obtained the stability estimates shown in Theorem \ref{thm:pri:1}. Due to the mixed-type nature of the problem, obtaining \textit{a-priori} estimates does not immediately result in the construction of a regular solution. This is a common feature of mixed-type problems. 

Regarding the construction of a solution, we point the reader towards the excellent recent work of Dalibard-Marbach-Rax, \cite{DMR}, which is closely related to the current paper. We also point the reader to the references \cite{Pagani1}, \cite{Pagani2} who studied a similar linear mixed-type problem to $y \p_x u - \p_y^2 u = 0$.  In \cite{DMR}, the authors consider the model problem $(\bar{u} + \eps u) \p_x u - \p_y^2 u = 0$, where $\bar{u}$ undergoes a change of sign. The authors importantly identify a series of constraints which need to be satisfied by the data on the sides in order to guarantee the existence of a sufficiently regular solution. In the present work, we use these constraints as a starting point towards the construction of our solution, Theorem \ref{thm:twomain}. However, we emphasize that apart from adopting the constraints derived in \cite{DMR}, the methods in this paper are considerably different due to the more complicated nature of the Prandtl system. Explaining our strategy will be the purpose of the next subsection.

\subsection{Overview of Strategy for Theorem \ref{thm:twomain}}

\paragraph{\textsc{Closing the ``Loop"}} Given \eqref{data:set}, the main difficulty is that the numbers in $\bold{A}_{k_\ast}$ depend on $G$, which, in the real problem depends on the solution itself. Therefore, our entire strategy is devoted to ensuring the ``loop" closes: the solution depends on the data, but now the data also depends back on the solution. In the end, we will show the following inequalities:
\begin{subequations} 
\begin{align} \label{abd}
\| \psi \|_{\mathcal{Z}} \le & \| F_{Left}, F_{Right} \|_{\mathcal{P}} \qquad && \text{(proven in \cite{IM22})} \\ \label{bbd}
|\bold{A}_{k_\ast} |_{\mathbb{R}^{4k_\ast + 1}}  \lesssim &o_L(1) \| \psi \|_{\mathcal{Z}} \qquad  &&\text{(proven in present paper)}. 
\end{align}
\end{subequations}
The sequence of bounds \eqref{abd} -- \eqref{bbd} closes due to the small $o_L(1)$ factor in \eqref{abd}. The main objective of this paper, is to therefore establish the second inequality, \eqref{bbd}, on the scalars appearing in $\bold{A}_{k_\ast}$.

\paragraph{\textsc{Perturbative Formulation}} For this discussion, we consider the case when $0 < L << 1$, which requires the majority of the effort to establish. In this case, our starting point is the system \eqref{loc:sys:1} in the Appendix, which we recall here:
\begin{subequations} \label{loc:sys:1:intro} 
\begin{align}
&Z \p_s \Omega_{k,I} - \p_Z^2 \Omega_{k,I} =   F_{k,I}, \qquad (s, Z) \in (1, 1 + \bar{L}) \times \mathbb{R} \\
&\Omega_I(s, \pm \infty) = 0,  \qquad s \in (1, 1+\bar{L}) 
\end{align}
\end{subequations}
This formulation requires taking advantage of several cancellations in the Prandtl system (through various changes of variables and changes of unknowns) which was performed in our stability paper, \cite{IM22}. Our starting point for this work is to introduce the rescalings: 
\begin{align}
t := \frac{s-1}{L}, \qquad \rho := \frac{Z}{L^{\frac13}}, 
\end{align} 
which, for $0 < L << 1$, correspond to a fast scale near the interface $\{Z = 0\} = \{\rho = 0\}$. Under these rescalings, we observe that \eqref{loc:sys:1:intro} transforms into the normalized domain 
\begin{subequations} \label{normdom1} 
\begin{align}
&\rho \p_t \Xi_{k} - \p_\rho^2 \Xi_{k} =  L^{\frac23} G_{k}, \qquad (t, \rho) \in (0, 1) \times \mathbb{R} \\
&\Xi_{k}(t, \pm \infty) = 0,  \qquad t \in (0, 1), \\
&\Xi_{Left}^{(k)}(\rho) := \Xi_k(0, \rho), \qquad \rho > 0, \\
&\Xi_{Right}^{(k)}(\rho) := \Xi_k(1, \rho), \qquad \rho < 0,
\end{align}
\end{subequations}
where $\Xi_k(t, \rho) = \Omega_{k,I}(s, Z)$ and $G_{k}(t, \rho) = F_{k,I}(s, Z)$. The presence of the scaling factor $L^{\frac23}$ appearing on the right-hand side above reflects the fact that the left-hand side gains $2/3$ tangential derivatives over the right-hand side. We also emphasize at this point that the highest order derivative terms in $G_k$ will be of the form $L^{-\frac13} \p_{\rho} \Xi_k$ and the ``quasilinear term" $L^{-\frac13} \rho \p_{\rho}^2 \Xi_k$. To treat the latter term, we will use that the weight $\rho$ vanishes at $\{\rho = 0\}$.

Another aspect of our formulation is that, instead of taking $\p_t$ of the equation \eqref{normdom1}, we will study \eqref{normdom1} $k$ by $k$. In other words, $\Xi_{k+1} \neq \p_t \Xi_k$, but rather 
\begin{align} \label{slashD}
\Xi_{k+1} =  \p_t \Xi_k +\Big( \frac{\p_t \bar{w}_z}{\bar{w}_z} \Xi_k+ \frac{1}{\bar{w}_z^2} \p_t \Big( \frac{\bar{w}_{zz}}{\bar{w}_z} \Big) u_k \Big) \chi\Big( \frac{\overline{L}^{\frac13}\rho}{\delta} \Big) = \slashed{D} \Xi_k
\end{align}
where we will have a use for the following ``error" term:
\begin{align} \label{defslashE}
\slashed{E}_k :=\Big(  \frac{\p_t \bar{w}_z}{\bar{w}_z} \Xi_k+ \frac{1}{\bar{w}_z^2} \p_t \Big( \frac{\bar{w}_{zz}}{\bar{w}_z} \Big) u_k\Big) \chi\Big( \frac{\overline{L}^{\frac13}\rho}{\delta} \Big).
\end{align}

\paragraph{\textsc{Intrinsic Constraints}} By re-applying the constraints in \eqref{const:const:I} to the system \eqref{loc:sys:1:intro}, we obtain for each $1 \le k \le k_{\ast}$, the following: 
\begin{subequations} \label{const:const:k}
\begin{align} \label{dg:c:k}
\ell_1[\Xi^{(k-1)}_L, \Xi^{(k-1)}_R] = & \beta_1(G^{(k-1)}_L, G^{(k-1)}_R) + \alpha_1(G_k),  \\  \label{dg:d:k}
\ell_0[\Xi^{(k-1)}_L, \Xi^{(k-1)}_R] = & \beta_0(G^{(k-1)}_L, G^{(k-1)}_R) + \alpha_0(G_k), 
\end{align}
\end{subequations}
The strategy is then to use the equation \eqref{loc:sys:1:intro} to obtain inductive relations 
\begin{align} \label{inductive:form:1}
\Xi^{(k+1)}_{Left}(\rho) = \frac{(\Xi^{(k)}_{Left})'' + L^{\frac23} G^{(k)}_{Left}}{\rho}, \qquad \Xi^{(k+1)}_{Right}(\rho) = \frac{(\Xi^{(k)}_{Right})'' + L^{\frac23} G^{(k)}_{Right}}{\rho}
\end{align}
to reformulate \eqref{const:const:k} into \textit{intrinsic} constraints on the two given data: $\Xi_{Left}, \Xi_{Right}$. This process is detailed and involved due to the dependence of the right-hand side, $G_k$ on linear (and nonlinear) terms depending on $\Xi$, and in turn this means $G^{(k)}_{Left}$ depends on $\Xi_{Left}$ and $G^{(k)}_{Right}$ depends on $\Xi_{Right}$. 

Therefore, instead of demanding a simple and explicit set of constraints (as shown in \eqref{const:const}), we extract the following implicit set of intrinsic constraints: 
\begin{subequations} \label{const:real}
\begin{align}
\ell_{1 + 3(k-1)}(\Xi_{L}, \Xi_{R}) = & \alpha_{1}(G_k) +  \beta_{1 + 3(k-1)}(G_L^{(k-1)}, G_R^{(k-1)}, \dots, G_L, G_R) \\ 
\ell_{0 + 3(k-1)}(\Xi_{L}, \Xi_{R}) = & \alpha_{0}(G_k) +  \beta_{0 + 3(k-1)}(G_L^{(k-1)}, G_R^{(k-1)}, \dots, G_L, G_R),
\end{align}
\end{subequations}
We need to explain two types of quantities appearing on the right-hand side above, which is the content of the forthcoming paragraphs. It is worth emphasizing that the form \eqref{const:real} emerges from the mechanics of inserting \eqref{inductive:form:1} into \eqref{dg:c:k} -- \eqref{dg:d:k}. However, our point-of-view is to now regard the functionals $\beta_{j}$ as having two components: primarily a dependence on $\Xi_L, \Xi_R$ itself (as these are local quantities evaluated on the tangential boundaries), but also a weak dependence on the interior solution due to the nonlocal term $\psi$.

We therefore rewrite \eqref{const:real} into the form
\begin{subequations} \label{const:real:2}
\begin{align}
\ell_{1 + 3(k-1)}(\Xi_{L}, \Xi_{R}) = & \mu_{k,1} +  \beta_{1 + 3(k-1)}(\Xi_L, \Xi_R), \\ 
\ell_{0 + 3(k-1)}(\Xi_{L}, \Xi_{R}) = & \mu_{k,0} +  \beta_{0 + 3(k-1)}(\Xi_L, \Xi_R),
\end{align}
\end{subequations}
where above $\mu_{k,1} = \alpha_1(G_k)$ and $\mu_{k,0} = \alpha_0(G_k)$.

\paragraph{\textsc{The $\mathcal{Z}$ norm}} We turn first to the presence of the numbers $\nu_k, \lambda_k, \mu_{k,0}, \mu_{k,1}$ on the right-hand side of \eqref{const:real}. These numbers are ``live" in the sense that that they are determined by the whole solution itself, through the integrals $L^{\frac23} \int \int_0^\infty \p_t G^{(k)}$ and $L^{\frac23} \int \int_0^\infty  \p_t G^{(k)} \rho$ appearing in \eqref{const:const:k}. Therefore, we need to estimate these quantities in the following way: 
\begin{align} \label{cont:nu}
 |\mu_{k,0}| + |\mu_{k,1}| \le o_L(1) \| \Omega \|_{\mathcal{Z}}, \qquad 1 \le k \le k_{\ast}
\end{align}
When $L << 1$, this gives rise to an iterative argument: the solution depends on the data, but the data also depends on the solution (in a weak way). We also would like to emphasize that the control \eqref{cont:nu} is highly nontrivial, and it utilizes the full strength of our powerful $\mathcal{Z}$ norm.

\paragraph{\textsc{Small Lipschitz Perturbations}} We turn next to the presence of the function $\beta(\cdot, \cdot)$ on the right-hand side of \eqref{const:const:k}. This will be a function (which can change from line to line) with a Lipschitz dependence on the data. More specifically, let $\vec{F} := (F_{Left}, F_{Right})$ and $\vec{G} := (G_{Left}, G_{Right})$. Then $\beta(\cdot, \cdot)$ satisfies for $\| F \|_{\mathcal{P}}, \| G \|_{\mathcal{P}} < \infty$ the bound 
\begin{align*}
|\beta(\vec{F}) - \beta(\vec{G})| \lesssim \| \vec{F} - \vec{G} \|_{\mathcal{P}},
\end{align*} 

To understand the method of satisfying the constraints \eqref{const:real}, we can consider a simpler toy problem. Suppose, for the sake of a toy problem, we wish to satisfy the constraint
\begin{align} \label{toy:con}
F_{Left}(0) = \nu  + \ell \int |F_{Left}'|^2 \ud \rho, 
\end{align}
where $\nu \in \mathbb{R}$ is a given number and $|\ell| = o_L(1) << 1$. The description of a set (a manifold) which satisfies this constraint can be achieved in the following manner. Fix a function $F_{\parallel}$ which spans the $\ell = 0$ constraint (so $F_{\parallel}(0) = 1$), and let $\overline{F} \in P^{\perp} := \{ F : F(0) = 0 \}$ be the $\ell = 0$ ``perpendicular" subspace. We will seek solutions to \eqref{toy:con} in the form
\begin{align}
F = \overline{F} +  \widetilde{\nu} F_{\parallel} =\overline{F} +  \widetilde{\nu}[\overline{F}; \ell] F_{\parallel}.   
\end{align}
Above, the modified number $\widetilde{\nu}$ depends on the choice of $\overline{F} \in P^\perp$ (hence we obtain a manifold), and moreover satisfies the bound 
\begin{align}
|\widetilde{\nu} - \nu| \lesssim o_L(1). 
\end{align}
To perform the construction, we can simply proceed iteratively using the smallness of $L << 1$. Indeed, inserting $\widetilde{\nu} = \nu + \ell \nu_1$, we can collect terms to find 
\begin{align} \n
\nu_1 = & \Big( \int |\overline{F}'|^2 + 2 \nu \overline{F}' F_{\parallel}' + \nu^2 |F_{\parallel}'|^2 \Big) + \ell \nu_1 \Big(\int \overline{F}' F_{\parallel}' + 2 \nu |F_{\parallel}'|^2 \Big)  + \ell^2 \nu_1^2 \int |F_{\parallel}'|^2. 
\end{align}
At this stage, it is possible to show the existence of a solution, $\nu_1$, to the above algebraic equation. Alternatively, one can choose the leading order for $\nu_1 = \Big( \int |\overline{F}'|^2 + 2 \nu \overline{F}' F_{\parallel}' + \nu^2 |F_{\parallel}'|^2 \Big)$. This creates an error of size $\ell^2$, and we can iterate this procedure for a modified guess $\widetilde{\nu} = \nu + \ell \nu_1 + \ell^2 \nu_2$.  

We observe that the key point which forces this iteration to converge is that the perturbation $\int |F_{Left}'|^2$ appearing in \eqref{toy:con} is Lipschitz with respect to a norm $\| \cdot \|_{\mathcal{P}}$, where the spanning vectors $\| \overline{F} \|_{\mathcal{P}}$ and $\| F_{\parallel} \|_{\mathcal{P}}$ are $O(1)$. 

\section{Results on Burgers Equation}

\subsection{$\p_t$ Regularity of Strong Solutions}

We first consider the model problem:
\begin{subequations} \label{modelpb1}
\begin{align}
&\rho \p_t \Omega - \p_\rho^2 \Omega = G, \qquad (t, \rho) \in (0, 1) \times \mathbb{R}\\
&\Omega|_{t = 0} = \Xi_{L}, \qquad \rho > 0 \\
&\Omega|_{t = 1} = \Xi_{R}, \qquad \rho < 0, \\
&\Omega|_{\rho \rightarrow \pm \infty} = 0. 
\end{align}
\end{subequations}
Here, the source term $G$ is considered a given forcing on which we have no choice. Eventually $G$ will be replaced with perturbative linear and nonlinear terms. For this problem, we have the existence of weak solutions, as stated in \cite{DMR}:
\begin{lemma} Let $G \in L^2((0, 1); H^{-1}(\mathbb{R}))$. There exists a unique strong solution $\Omega \in C^0((0, 1); H^1(\mathbb{R}))$ to \eqref{modelpb1}. 
\end{lemma}

We now want to provide constraints on the data, $(\Xi_{L}, \Xi_{R})$, so as to be able to upgrade the solution by one tangential derivative, $\p_t$, of regularity. We have the following result, which is inspired by Proposition 2.7 in \cite{DMR} (but re-formulated for our setting).
\begin{proposition} \label{Propcon} $G \in H^1((0, 1); L^2(\mathbb{R}))$ There exist nontrivial, independent linear functionals $\ell_1, \ell_0: \mathcal{P} \mapsto \mathbb{R}$, and linear functionals $\beta_0, \beta_1$ such that if the following constraints hold
\begin{subequations} \label{const:const}
\begin{align} \label{con:a}
\Xi_{L}''(0) = &G_{Left}(0), \\ \label{con:b}
\Xi_{R}''(0) = & G_{Right}(0), \\  \label{con:c}
\ell_1(\Xi_{L}, \Xi_{R}) = & \beta_1(G) \\ \label{con:d}
\ell_0(\Xi_{L}, \Xi_{R}) = & \beta_0(G),
\end{align}
\end{subequations}
then the unique weak solution to \eqref{modelpb1} satisfies $\Omega \in H^1((0, 1); H^1(\mathbb{R}))$.
\end{proposition}
First, we define 
\begin{align}
\Xi^{(k)}_{L} := \p_t^k \Omega|_{x = 0}. 
\end{align}
We clearly see that 
\begin{align} \label{para:con:1}
\Xi^{(1)}_{L} = \frac{\p_\rho^2 \Xi_{L} + G_L}{\rho}, \qquad \Xi^{(1)}_{R} = \frac{\p_\rho^2 \Xi_{R} + G_R}{\rho}
\end{align}
and so on for higher $k$. We define $G_L, G_R$ in the natural way: evaluation at $t = 0, \rho > 0$ and $t = 1, \rho < 0$. Note that due to these expressions are well-defined precisely due to the first two constraints \eqref{con:a} -- \eqref{con:b}.


We want to think of the two local constraints \eqref{con:a} -- \eqref{con:b} separately from the two integral constraints \eqref{con:c} -- \eqref{con:d}. To achieve this, we will decompose our data into 
\begin{align} \label{QLQR}
\Xi_L = \widehat{\Xi}_L + Q_L, \qquad \Xi_R = \widehat{\Xi}_R + Q_R,
\end{align}
where $(\widehat{\Xi}_L, \widehat{\Xi}_R)$ are compactly supported \textit{away} from $\rho = 0$, and 
\begin{align}
Q_L''(0) = &G_{Left}(0), \qquad Q_R''(0) = G_{Right}(0).
\end{align}

We define now the dual profiles following \cite{DMR}. First, we have 
\begin{align}
&- \rho \p_t \Phi^0 - \p_\rho^2 \Phi^0 = 0 \\
&\Phi^0 \text{ continuous across } \rho = 0 \\
&\p_\rho \Phi^0(t, 0+) - \p_\rho \Phi^0(t,0-) = -1 \\
&\Phi^0|_{t = 1} = 0, \qquad \rho > 0 \\
&\Phi^0|_{t = 0} = 0, \qquad \rho < 0, \\
&\Phi^0 \rightarrow 0 \text{ as } |\rho| \rightarrow \infty
\end{align}
From here, we define the notations 
\begin{align}
\Phi^0_L(\rho) := \Phi^0(0, \rho), \qquad \rho > 0, \qquad \Phi^0_R(\rho) := \Phi^0(1, \rho), \qquad \rho < 0. 
\end{align}
Notice that $\Phi^0_{L,R}$ are not prescribed functions. However, we can deduce the following from the jump constraints: 
\begin{align}
\Phi^0_L(0)  &= 0\\
\p_z \Phi^0_L(0) &= -1\\
\Phi^0_R(0) &= 0\\
\p_z \Phi^0_R(0) &= 1 
\end{align}

Similarly, we define $\Phi^1$ via 
\begin{align}
&- \rho \p_t \Phi^1 - \p_\rho^2 \Phi^1 = 0 \\
&\p_\rho \Phi^1 \text{ continuous across } \rho = 0 \\
&\Phi^1(t, 0+) -  \Phi^1(t,0-) = 1 \\
&\Phi^1|_{t = 1} = 0, \qquad \rho > 0 \\
&\Phi^1|_{t = 0} = 0, \qquad \rho < 0, \\
&\Phi^1 \rightarrow 0 \text{ as } |\rho| \rightarrow \infty
\end{align}
From here, we define the notations 
\begin{align}
\Phi^1_L(\rho) := \Phi^1(0, \rho), \qquad \rho > 0, \qquad \Phi^1_R(\rho) := \Phi^1(1, \rho), \qquad \rho < 0. 
\end{align}
Again, $\Phi^1_{L,R}$ are not prescribed functions. However, we can deduce the following from the jump constraints: 
\begin{align}
\Phi^1_L(0)  &= 1\\
\p_z \Phi^1_L(0) &= 0\\
\Phi^1_R(0) &= -1\\
\p_z \Phi^1_R(0) &= 0  
\end{align}

We now examine the constraint equations which appear in Proposition 2.7 of \cite{DMR}:
\begin{align} \label{con0}
\int_0^\infty \rho \Phi^0_L \Xi^{(1)}_L - \int_{-\infty}^0 \rho \Phi^0_R \Xi^{(1)}_R + \Xi_L(0) - \Xi_R(0) = - \int_{\Omega} \p_t G \Phi^0 \\  \label{con1}
\int_0^\infty \rho \Phi^1_L \Xi^{(1)}_L - \int_{-\infty}^0 \rho \Phi^1_R \Xi^{(1)}_R + \Xi_L'(0) - \Xi_R'(0) = - \int_{\Omega} \p_t G \Phi^1
\end{align}
We insert expression \eqref{para:con:1} to write this as an intrinsic constraint containing only $(\Xi_L, \Xi_R)$ and the given source term $G$. Doing so, we collect the identities 
\begin{align}
\int_0^\infty \rho \Phi^0_L \Xi^{(1)}_L = &\int_0^\infty \rho \Phi_L^0 \frac{\p_\rho^2 \Xi_{L} + G_L}{\rho} = \int_0^\infty \Phi_L^0 \p_\rho^2 \widehat{\Xi}_{L} + \int_0^\infty \Phi_L^0 (\p_{\rho}^2 Q_L + G_L) \\
\int_{-\infty}^0 \rho \Phi^0_R \Xi^{(1)}_R = & \int_{-\infty}^0 \rho \Phi^0_R \frac{\p_\rho^2 \Xi_R + G_R}{\rho} = \int_{-\infty}^0 \Phi^0_R \p_\rho^2 \widehat{\Xi}_R + \int_{-\infty}^0 \Phi^0_R (\p_{\rho}^2 Q_R + G_R).
\end{align}
We can therefore rewrite \eqref{con0} -- \eqref{con1} by moving all the $(\widehat{\Xi}_L, \widehat{\Xi}_R)$ terms to the left and the $G$ dependent terms to the right as follows
\begin{subequations}
\begin{align}\n
&\int_0^\infty \Phi^0_L \widehat{ \Xi}_L'' - \int_{-\infty}^0 \Phi^0_R \widehat{\Xi}_R'' \\  \label{con:com:a}
 =& - \int \p_t G \Phi^0 - \int_0^\infty \Phi^0_L (G_L + Q_L'' ) + \int_{-\infty}^0 \Phi^0_R (G_R + Q_R'') - (Q_L(0) -Q_R(0))\\ \n
&\int_0^\infty \Phi^1_L \widehat{\Xi}_L'' - \int_{-\infty}^0 \Phi^1_R \widehat{\Xi}_R''  \\ \label{con:com:b} 
=& - \int \p_t G \Phi^1 - \int_0^\infty \Phi^1_L  (G_L + Q_L'') + \int_{-\infty}^0 \Phi^1_R (G_R + Q_R'') - (Q'_L(0) -Q'_R(0))
\end{align}
\end{subequations}
We rewrite this in short-form as 
\begin{align} \label{lf1}
\ell_0[\widehat{\Xi}_L,\widehat{ \Xi}_R] = \beta_0(G_L, G_R) + \alpha_0(G), \\ \label{lf2}
\ell_1[\widehat{\Xi}_L, \widehat{\Xi}_R] = \beta_1(G_L, G_R) + \alpha_1(G). 
\end{align}
where $\ell_0, \ell_1, \beta_0, \beta_1, \alpha_0, \alpha_1$ are linear functionals taking values in $\mathbb{R}$, and are defined through \eqref{con:com:a} -- \eqref{con:com:b}: 
\begin{align}
\beta_0(G_L, G_R) := & - \int_0^\infty \Phi^0_L (G_L +Q_L'') + \int_{-\infty}^0 \Phi^0_R (G_R + Q_R'')- (Q_L(0) -Q_R(0)) \\
\beta_1(G_L, G_R) := & - \int_0^\infty \Phi^1_L (G_L + Q_L'') + \int_{-\infty}^0 \Phi^0_R G_R + Q_R'')- (Q_L'(0) -Q_R'(0))\\
\alpha_0(G) := &- \int \p_t G \Phi^0, \\
\alpha_1(G) := &- \int \p_t G \Phi^1. 
\end{align}

Our goal at this stage is to prove that $\ell_0, \ell_1$ are not identically zero (which implies that the Range of $\ell_0, \ell_1$ is all of $\mathbb{R}$). Assuming sufficient regularity of the dual profiles (which is justified on the support of $\widehat{\Xi}_L, \widehat{\Xi}_R$), we can integrate by parts via to rewrite $\ell_0, \ell_1$ as  
\begin{align} \label{form:ell:0}
\ell_0[\widehat{\Xi}_L, \widehat{\Xi}_R] = &\int_0^\infty \p_\rho^2 \Phi^0_L \widehat{\Xi}_L -  \int_{-\infty}^0 \p_\rho^2 \Phi^0_R \widehat{\Xi}_R, \\  \label{form:ell:1}
\ell_1[\widehat{\Xi}_L, \widehat{\Xi}_R] =& \int_0^\infty \p_\rho^2 \Phi^1_L \widehat{ \Xi}_L -  \int_{-\infty}^0 \p_\rho^2 \Phi^1_R \widehat{\Xi}_R.
\end{align}
Due to the boundary conditions at $0$ and, importantly, the vanishing at $|\rho| = \infty$, we know that $\p_\rho^2 \Phi^0, \p_\rho^2 \Phi^1$ are nontrivial, $L^2$ functions. Therefore, $\ell_0, \ell_1$ are nontrivial. 

Fix pairs $\bold{e}_0 = (e_{0,L}, e_{0,R})$, $\bold{e}_1 = (e_{1,L}, e_{1,R})$ that are smooth, compactly supported away from $\rho = 0$, and such that 
\begin{align}
&\ell_0[\bold{e}_0] = 1, \qquad \ell_0[\bold{e}_1] = 0, \\
&\ell_1[\bold{e}_0] = 0, \qquad \ell_1[\bold{e}_1] = 1.
\end{align}
We can make such a choice due to the pairwise independence of the functionals. We have therefore established 
\begin{proposition} \label{prop:k:eq:1} Let $G$ be a given source term in \eqref{modelpb1}. Fix $(Q_L, Q_R)$ as in \eqref{QLQR}. Fix any $(\bar{\Xi}_L, \bar{\Xi}_R) \in \bar{\mathcal{P}}_1$ There exist two numbers $c_0, c_1$ (which depend on $G$ and $(\bar{\Xi}_L, \bar{\Xi}_R)$) such that data of the type 
\begin{align} \label{data:set}
(\Xi_L, \Xi_R) = & (\widehat{ \Xi}_L, \widehat{\Xi}_R) + (Q_L, Q_R) \\ 
 (\widehat{ \Xi}_L, \widehat{\Xi}_R) = & (\bar{\Xi}_L, \bar{\Xi}_R)  + c_0 \bold{e}_0 + c_1 \bold{e}_1
\end{align}
satisfy the constraint equations \eqref{const:const}. Therefore, by Proposition 2.7 in \cite{DMR}, the unique strong solution resulting from \eqref{data:set} enjoys the higher tangential regularity $H^1(0, 1; H^1(\mathbb{R}))$. 
\end{proposition}

\subsection{Linear Independence of $\p_t^k \Phi^{(j)}$, $k \ge 1$}

We need the following rigidity result, which is in the spirit of control theory: 
\begin{lemma} \label{lem:rig:1} Let $\eps > 0$. Assume $f$ is a solution to 
\begin{subequations} \label{eq:f:f}
\begin{align}
&(\rho \p_t +  \p_{\rho}^2) f = 0, \qquad (t, \rho) \in (0, 1) \times (\eps, \infty), \\
&f|_{t = 0} = 0, \qquad \rho > \eps, \\
&f|_{t = 1} = 0, \qquad \rho > \eps, \\
&f|_{\rho \rightarrow \infty} = 0. 
\end{align}
\end{subequations}
Then for all $f = 0$ for $(t, \rho) \in (0, 1) \times (\eps, \infty)$.  
\end{lemma}
\begin{proof} Define the domain $\Omega_\eps := (t, \rho) \in (0, 1) \times (\eps, \infty)$. Let the boundary condition 
\begin{align}
g(t) := f|_{\rho = \eps}.
\end{align}

\noindent \textit{Step 1: } We multiply by $(\rho - \eps)$. This gives 
\begin{align*}
0 = \int \int_{\Omega_\eps}  \rho \p_t f ( \rho - \eps) \ud t \ud \rho + \int \int_{\Omega_\eps} \p_\rho^2 f ( \rho - \eps) \ud \rho \ud t =  \int_{\rho = \eps} f \ud t =  \int_0^1 g(t) \ud t
\end{align*}

\vspace{1 mm}

\noindent \textit{Step 2: } We now define $\alpha(\rho) := \int_0^1 f(t, \rho) \ud t$. We integrate in $t$ the equation \eqref{eq:f:f} for each fixed $\rho$, which produces 
\begin{subequations}
\begin{align}
\p_{\rho}^2 \alpha(\rho) = & 0 \qquad \rho \in (\eps, \infty), \\
 \alpha(\eps) = & 0 \\
 \alpha(\infty) = & 0. 
\end{align}
\end{subequations}
Therefore, we conclude $\alpha(\rho) = 0$. 

\vspace{1 mm}

\noindent \textit{Step 3:} We want to compute higher moments. We multiply by $(\rho - \eps) t$. This gives 
\begin{align*}
0 = &\int \int_{\Omega_\eps}  \rho \p_t f ( \rho - \eps)t \ud t \ud \rho + \int \int_{\Omega_\eps} \p_\rho^2 f ( \rho - \eps)t \ud \rho \ud t \\
 = & - \int \int_{\Omega_\eps} \rho ( \rho - \eps) f \ud t \ud \rho + \int_{\rho = \eps} f t \ud t \\ 
 = & - \int \rho (\rho - \eps) \alpha(\rho) \ud \rho + \int_0^1 t g(t) \ud t \\
 = & \int_0^1 t g(t) \ud t
\end{align*}
where we have used the vanishing of $\alpha$. 

\vspace{1 mm}

\noindent \textit{Step 4:} We define $\alpha^{(1)}(\rho) = \int_0^1 t f(t, \rho) \ud t$. Computing the equation on $t f$, we obtain 
\begin{align}
\rho \p_t (t f) + \p_{\rho}^2 (t f) - \rho f = 0. 
\end{align}
Integrating in $t$, we obtain 
\begin{subequations}
\begin{align}
0 = & \p_{\rho}^2 \alpha^{(1)} - \alpha =  \p_{\rho}^2 \alpha^{(1)}, \qquad \rho \in (\eps, \infty), \\
 \alpha^{(1)}(\eps) = & 0 \\
 \alpha^{(1)}(\infty) = & 0. 
\end{align}
\end{subequations}
Therefore, we have $\alpha^{(1)} = 0$. 

\noindent \vspace{1 mm} \textit{Step 5:} Bootstrapping up to higher moments, we obtain 
\begin{align}
\int_0^1 t^n g(t) = 0, \qquad n \in \mathbb{N} \cup \{0 \}. 
\end{align}
By the Stone-Weierstrauss Theorem, this implies that $g = 0$. 
\end{proof}

We now want to apply this result to arbitrary linear combinations of $\p_t^k \Phi$ to establish that these must be linearly independent on $\{0 \} \times (\eps, \infty)$ and $\{1 \} \times (- \infty, - \eps)$. 

\begin{lemma} Fix $j = 0$ or  $1$, and fix any $\eps > 0$. The collection
\begin{align}
\mathcal{C}^{(j)}_{k_\ast} := \{ \p_t^k \Phi^{(j)} \}_{k = 1}^{k_\ast}
\end{align}
is linearly independent and nontrivial on $\{0 \} \times (\eps, \infty)$ as well as on $\{1\} \times (-\infty, \eps)$.
\end{lemma}
\begin{proof} Assume there exist constants $C_k$, not all zero, such that
\begin{align}
\sum_{k = 1}^{k_\ast} C_k \p_t^k \Phi^{(j)}(0, \rho) = 0, \qquad \rho \in (\eps, \infty).
\end{align}
Define 
\begin{align}
f(t, \rho) = \sum_{k = 1}^{k_\ast} C_k \p_t^k \Phi^{(j)}(t, \rho), \qquad (t, \rho) \in \Omega_\eps. 
\end{align}
Observe that $f$ satisfies \eqref{eq:f:f}, and therefore by Lemma \ref{lem:rig:1} we have $f = 0$ on $\Omega_\eps$. Let $k_{min}$ be the minimal $k$ such that $C_k \neq 0$. We therefore have the following ODE in $t$ for each fixed $\rho \in (\eps, \infty)$: 
\begin{subequations}
\begin{align}
&0 = f(t, \rho) = \sum_{k = k_{min}}^{k_\ast} C_k \p_t^k \Phi^{(j)}(t, \rho), \qquad (t, \rho) \in \Omega_\eps \\
&\p_t^{k_{min}} \Phi^{(j)}(1, \rho) = 0, \qquad \rho \in (\eps, \infty),
\end{align}
\end{subequations}
which implies that 
\begin{align}
\p_t^{k_{min}} \Phi^{(j)} = 0 \qquad (t, \rho) \in \Omega_\eps.
\end{align}
By further invoking the boundary conditions for $\p_t^{k} \Phi^{(j)}|_{t = 1} = 0$, $1 \le k \le k_{min}$, we can further deduce that 
\begin{align}
\p_t \Phi^{(j)} = 0 \qquad (t, \rho) \in \Omega_\eps.
\end{align}
This is a contradiction to the definition of $\Phi^{(j)}$. 
\end{proof}

We now want to enhance the linear dependence to include the union of both collections above. To do so, we need to introduce the notion of a \textit{degree} (more precisely, two degrees, $d^{(0)}, d^{(1)}$). We now motivate this notion by considering two examples. 
\begin{example} $\Phi^{(0)}, \Phi^{(1)}$ cannot satisfy the equation 
\begin{align} \label{ex1}
\p_t^2 \Phi^{(1)} = \p_t \Phi^{(1)} + \Phi^{(0)}, \qquad (t, \rho) \in \{ (0, 1) \times (\eps, \infty) \} \cup  \{ (0, 1) \times (-\infty, \eps) \} 
\end{align}
for every $\eps > 0$.  
\end{example}
\begin{proof} To see this, we integrate twice from $t = 1$, $\rho > \eps$ to obtain 
\begin{subequations}
\begin{align}
\Phi^{(1)} = &- \int_t^1 \Phi^{(1)} + \int_t^1 \int_s^1 \Phi^{(0)}, \qquad \rho > \eps \\
\Phi^{(1)} = &  \int_0^t \Phi^{(1)} + \int_0^t \int_0^s \Phi^{(0)}, \qquad \rho < - \eps
\end{align}
\end{subequations}
We now recall that $\Phi^{(1)}$ is continuous across $z = 0$, whereas $\Phi^{(0)} = 1_+ + \overline{\Phi}^{(0)}$. Inserting this latter formula above we obtain 
\begin{subequations}
\begin{align} \n
\Phi^{(1)} = &- \int_t^1 \Phi^{(1)} +  \int_t^1 \int_s^1 1  + \int_t^1 \int_s^1 \overline{\Phi}^{(0)}, \\
=& - \int_t^1 \Phi^{(1)} - \frac{(1-t)^2}{2}  + \int_t^1 \int_s^1 \overline{\Phi}^{(0)}, \qquad \rho > \eps \\
\Phi^{(1)} = & \int_0^t \Phi^{(1)} + \int_0^t \int_0^s \overline{\Phi}^{(0)}, \qquad \rho < - \eps
\end{align}
\end{subequations}
By matching the limit as $\rho \downarrow 0$ to the limit as $\rho \uparrow 0$, we obtain 
\begin{align}
 - \int_t^1 \Phi^{(1)} - \frac{(1-t)^2}{2}  + \int_t^1 \int_s^1 \overline{\Phi}^{(0)} =  \int_0^t \Phi^{(1)} + \int_0^t \int_0^s \overline{\Phi}^{(0)}
\end{align}
By re-arranging and using the identity 
\begin{align*}
 \int_0^t \int_0^s \overline{\Phi}^{(0)} - \int_t^1 \int_s^1 \overline{\Phi}^{(0)} = & \int_0^t \int_0^s \overline{\Phi}^{(0)}  - \int_t^1 \int_0^1 \overline{\Phi}^{(0)} + \int_t^1 \int_0^s \overline{\Phi}^{(0)} \\
 = & \int_0^1 \int_0^s \overline{\Phi}^{(0)}- \int_t^1 \int_0^1 \overline{\Phi}^{(0)} \\
 = & \int_0^1 \int_0^s \overline{\Phi}^{(0)} - (1 - t)\int_0^1 \overline{\Phi}^{(0)},
\end{align*}
we obtain 
\begin{align}
- \frac{(1-t)^2}{2}  =  \int_0^1 \Phi^{(1)} +  \int_0^1 \int_0^s \overline{\Phi}^{(0)} - (1 - t)\int_0^1 \overline{\Phi}^{(0)},
\end{align}
which cannot hold for all $t$ due to the presence of the quadratic term on the left-hand side. 
\end{proof}

This example motivates the notion of $d^{(0)}$, whose purpose it is to count the degree of the polynomial (in $t$) contributed by each term above. 

\begin{definition} Fix any $k \ge 0$. The quantity $d^{(0)}$, called the degree of order zero, is defined as follows
\begin{align}
d^{(0)}[\p_t^{-k}  \Phi^{(0)}] := &k, \\
d^{(0)}[\p_t^{-k}  \Phi^{(1)}] := &k-1.
\end{align}
\end{definition}
The degree $d^{(0)}$ formalizes the counting in the following way. Rewrite (formally) \eqref{ex1} as follows 
\begin{align}
\Phi^{(1)} = \p_t^{-1} \Phi^{(1)} + \p_t^{-2} \Phi^{(0)}.
\end{align}
Applying $d^{(0)}$ to the right-hand side, we get that $d^{(0)}[ \p_t^{-2} \Phi^{(0)}]$ is $2$, which indicates a quadratic term is contributed, whereas $d^{(0)}[ \p_t^{-1} \Phi^{(1)}]$ which indicates that a zero-th order term is contributed. 

We now consider a different example, for which $d^{(0)}$ will be insufficient to obtain a contradiction, therefore motivating a second degree, $d^{(1)}$.
\begin{example}$\Phi^{(0)}, \Phi^{(1)}$ cannot satisfy the equation 
\begin{align} \label{ex2}
\p_t^2 \Phi^{(1)} = \p_t \Phi^{(0)} + \Phi^{(1)}, \qquad (t, \rho) \in \{ (0, 1) \times (\eps, \infty) \} \cup  \{ (0, 1) \times (-\infty, \eps) \} 
\end{align}
for every $\eps > 0$.  
\end{example}
\begin{proof} We again integrate to form the following two identities 
\begin{subequations}
\begin{align}
\Phi^{(1)} = &- \int_t^1 \Phi^{(0)} + \int_t^1 \int_s^1 \Phi^{(1)}, \qquad \rho > \eps \\
\Phi^{(1)} = &  \int_0^t \Phi^{(0)} + \int_0^t \int_0^s \Phi^{(1)}, \qquad \rho < - \eps.
\end{align}
\end{subequations}
In this case, if we are to compare the values of the upper and lower limit at $\rho = 0$, we will not be able to derive a polynomial expression with a \textit{unique} highest order term. This is reflected in the fact that $d^{(0)}$ applied to $\p_t^{-1} \Phi^{(0)}$ and $\p_t^{-2} \Phi^{(1)}$ yield the same value. Instead, we have to compute 
\begin{align}
1 = &\lim_{\rho \downarrow 0} \p_{\rho} \Phi^{(1)} - \lim_{\rho \uparrow 0} \p_{\rho} \Phi^{(1)}.
\end{align}
Differentiating the above expressions, we get 
\begin{subequations}
\begin{align}
\p_{\rho}\Phi^{(1)} = &- \int_t^1 \p_{\rho} \Phi^{(0)} + \int_t^1 \int_s^1 1 +\int_t^1 \int_s^1 \p_{\rho} \overline{\Phi}^{(1)} , \qquad \rho > \eps \\
\p_{\rho} \Phi^{(1)} = &  \int_0^t \p_{\rho} \Phi^{(0)} + \int_0^t \int_0^s \p_{\rho} \overline{\Phi}^{(1)}, \qquad \rho < - \eps.
\end{align}
\end{subequations}
Again, we see that the \textit{unique} quadratic term is contributed by the  $\int_t^1 \int_s^1 1$, whereas all other terms are linear or constant in $t$. This is a contradiction.
\end{proof}

To formalize the counting from this example, we have 
\begin{definition} Fix any $k \ge 0$. The quantity $d^{(1)}$, called the degree of order one, is defined as follows
\begin{align}
d^{(1)}[\p_t^{-k}  \Phi^{(0)}] := &k-1, \\
d^{(1)}[\p_t^{-k}  \Phi^{(1)}] := &k.
\end{align}
\end{definition}

Clearly, we can see that if $d^{(0)}[\p_t^{-k_0} \Phi^{(0)}] = d^{(0)}[\p_t^{-k_1} \Phi^{(1)}]$, this implies that $k_0 = k_1 -1$. In this case, the $d^{(1)}$ degree will not match. This counting procedure can be used to prove linear independence of the full collection, as we demonstrate below. 
\begin{lemma} Fix any $\eps > 0$. The collection
\begin{align}
\mathcal{C}_{k_\ast} = \mathcal{C}^{(0)}_{k_\ast}  \cup \mathcal{C}^{(1)}_{k_\ast}
\end{align}
is linearly independent and nontrivial on $\{0 \} \times (\eps, \infty)$ as well as on $\{1\} \times (-\infty, \eps)$.
\end{lemma}
\begin{proof}Assume there exist constants $C_k$, $D_k$ such that at least one $C_k \neq 0$ and at least one $D_k \neq 0$ and such that  
\begin{align}
\sum_{k = 1}^{k_\ast} C_k \p_t^k \Phi^{(0)}(0, \rho) + \sum_{k = 1}^{k_\ast} D_k \p_t^k \Phi^{(1)}(0, \rho) = 0, \qquad \rho \in (\eps, \infty).
\end{align}
Once again by appealing to Lemma \ref{lem:rig:1}, we can conclude that 
\begin{align}
\sum_{k = 1}^{k_\ast} C_k \p_t^k \Phi^{(0)}(t, \rho) + \sum_{k = 1}^{k_\ast} D_k \p_t^k \Phi^{(1)}(t, \rho) = 0, \qquad (t, \rho) \in \Omega_\eps.
\end{align}
Let $k_{max}$ be the largest index for which either $C_k, D_k \neq 0$. Without loss of generality, we can assume that $D_{k_{max}} = 1$. In this case, we have
\begin{align}
\p_t^{k_{max}} \Phi^{(1)}(t, \rho) + \sum_{k = 1}^{k_{max}-1} D_k \p_t^k \Phi^{(1)}(t, \rho) = - \sum_{k =1}^{k_{max}} C_k \p_t^k \Phi^{(0)}(t, \rho),
\end{align}
for $(t, \rho) \in (0, 1) \times (\eps, \infty)$ and $(t, \rho) \in (0, 1) \times (-\infty, -\eps)$. We now compute $d^{(0)}[\p_t^{-k_{max} }]$ and $d^{(1)}[\p_t^{-k_{max} }]$ of the expression on the right-hand side, which results in a contradiction. 
\end{proof}

\subsection{Structure of Data for Higher Derivatives}

We will define 
\begin{align}
\Xi_L = \widehat{\Xi}_L + Q_L, \qquad \Xi_R = \widehat{\Xi}_R + Q_R.
\end{align}
We will choose $Q_L, Q_R$ to be smooth, compactly supported functions satisfying the local relations 
\begin{align} \label{QLQR}
\p_{\rho}^{2 + 3k} Q_L(0) = - \p_{\rho}^{3k} G_L(0), \qquad \p_{\rho}^{2 + 3k} Q_R(0) = - \p_{\rho}^{3k} G_R(0), \qquad 0 \le k \le k_\ast -1. 
\end{align}
We will then choose $\widehat{\Xi}_L, \widehat{\Xi}_R \in C^\infty_c(\mathbb{R}_+) \times C^\infty_c(\mathbb{R}_+)$ (that is, supported \textit{away} from $\rho = 0$ and also compactly supported in $\rho$). We are thus led to consider the system 
\begin{align}
\ell_0^{(k)}[\widehat{\Xi}_L, \widehat{\Xi}_R] = \tau_0^{(k)}[G], \qquad 1 \le k \le k_\ast, \\
\ell_1^{(k)}[\widehat{\Xi}_L, \widehat{\Xi}_R] = \tau_1^{(k)}[G], \qquad 1 \le k \le k_\ast, 
\end{align}
where due to the compact support away from $\rho = 0$ of $\widehat{\Xi}_L, \widehat{\Xi}_R$, we can write 
\begin{align}
\ell_0^{(k)}[\widehat{\Xi}_L, \widehat{\Xi}_R] = & \int_0^\infty \rho \p_t^{k} \Phi^{(0)}_L \widehat{\Xi}_L - \int_0^\infty \rho \p_t^{k} \Phi^{(0)}_R \widehat{\Xi}_R \\
\ell_1^{(k)}[\widehat{\Xi}_L, \widehat{\Xi}_R] =  & \int_0^\infty \rho \p_t^{k} \Phi^{(1)}_L \widehat{\Xi}_L - \int_0^\infty \rho \p_t^{k} \Phi^{(1)}_R \widehat{\Xi}_R.
\end{align}
We have therefore established the following proposition. 
\begin{proposition} Choose any $(\bar{\Xi}_L, \bar{\Xi}_R) \in \bar{\mathcal{P}}_{k_\ast}$. Fix a pair $(Q_L, Q_R)$ of smooth, compactly supported functions satisfying  \eqref{QLQR}. There exist $2k_\ast$ numbers such that depend on $(\bar{\Xi}_L, \bar{\Xi}_R)$ and $G$, such that the following is true. Let 
\begin{align}
(\Xi_L, \Xi_R) = (\widehat{\Xi}_L, \widehat{\Xi}_R) + (Q_L, Q_R), 
\end{align}
 where
\begin{align}
(\widehat{\Xi}_L, \widehat{\Xi}_R) = (\bar{\Xi}_L, \bar{\Xi}_R) + \sum_{k = 1}^{k_\ast} (c_0^{(k)} \bold{e}_0^{(k)} + c_1^{(k)} \bold{e}_1^{(k)})
\end{align}
Then $(\Xi_L, \Xi_R)$ give rise to a $k^\ast$ differentiable solution of \eqref{modelpb1} in $t$. 
\end{proposition}

\section{The Prandtl equations near the FS profiles}

In this section, for the sake of being self-contained, we follow the derivations of \cite{IM22}, as the formulations we work with are overlapping. 

\subsection{The $(x, y)$ Coordinates}

 We start naturally at the perturbative formulation \eqref{remsum1}, which we insert into the Prandtl equations \eqref{eq:PR:0}, and produce the following equations on the remainders, $\psi_R$:
\begin{align}
(u_{FS} + \eps u_R) \p_x u_R + u_R \p_x u_{FS} + (v_{FS} + \eps v_R) \p_y u_R + v_R \p_y u_{FS} - \p_y^2 u_R = 0
\end{align}
We will have use for the tangentially differentiated version of these equations. We apply scaled derivatives by differentiating with respect to $\tau$, which is the normalized tangential scale: 
\begin{align}
\tau := \frac{x-1}{L}, \qquad \tau \in (0, 1). 
\end{align}
To compactify notation, we define 
\begin{align}
u_{R,k} := \p_{\tau}^k u_R, \qquad v_{R,k} := \p_{\tau}^k v_R.  
\end{align}
These differentiated unknowns satisfy the following equation 
\begin{align} \label{sideLH}
&u_P \p_x u_{R, k} + u_{Py} v_{R,k} + \bold{C}_1  \p_y u_{R,k} + \bold{C}_2 u_{R,k}  +\bold{C}_3 \psi_{R,k} - \p_y^2 u_{R, k}= \mathcal{R}_k,
\end{align}
where the coefficients above are defined to be 
\begin{subequations} \label{boldBigC}
\begin{align}
\bold{C}_1 := &v_{FS} + \eps 1_{k \ge 1} v_R\\
\bold{C}_2 := &\p_x u_{FS} + 1_{k \ge 1} k \p_x u_{P} + \eps 1_{k \ge 2} \p_x u_R \\
\bold{C}_3 := & 1_{k = 1} \p_{xy} u_{FS} + 1 _{k \ge 2} k  \p_{xy} u_P, 
\end{align}
\end{subequations}
and the remainder is given by 
and where
\begin{align} \n
\mathcal{R}_k = & -1_{k \ge 3} \sum_{k' = 1}^{k-2} \binom{k}{k'} u_{FS,k-k'} \p_x u_{R, k'} -1_{k \ge 3} \eps \sum_{k' = 1}^{k-2} \binom{k}{k'} u_{R,k-k'} \p_x u_{R, k'}  \\ \n
&- 1_{k \ge 1} \sum_{k' = 0}^{k-1} u_{R,k'} \p_x u_{FS, k-k'} - 1_{k \ge 1} \sum_{k' = 0}^{k-1} \binom{k}{k'} v_{FS, k-k'} \p_y u_{R,k'} \\ \n
&- 1_{k \ge 3} \sum_{k' = 1}^{k-2} \binom{k}{k'} v_{R,k'} \p_y u_{FS,k-k'}  - 1_{k \ge 3} \eps \sum_{k' = 1}^{k-2} \binom{k}{k'} v_{R,k'} \p_y u_{R,k-k'}  \\ \label{defRKRK}
& - 1_{k \ge 2} v_R \p_{xy} u_{FS,k-1} - 1_{k \ge 2} u_{FS,k} \p_x u_R.
\end{align}

\subsection{The $(s, z)$ Coordinates}

We introduce the variables, $(s, z)$, via
\begin{align} \label{z:variable}
\frac{\ud s}{\ud x} = \frac{1}{\Lambda(x)^2}, \qquad z := \frac{y}{\Lambda(x)} = \frac{y}{\Lambda_G(x) + \eps \Xi(x)},
\end{align}
where the ``free-boundary" $\Xi(x)$ is defined via 
\begin{align}
u_P(x, \Lambda(x)) = u_{FS}(x, \Lambda_G(x) + \eps \Xi(x)) + \eps u_R(x, \Lambda_G(x) + \eps \Xi(x))= 0. 
\end{align}
We now transform our functions to 
\begin{subequations}
\begin{align} \label{change:function}
&u_k(s, z) = u_{R, k}(x, y), \qquad v_k(s, z) := - \int_0^z \p_s u_k, \qquad \psi_k(s, z) := \int_0^z u_k \\ 
&\bar{u}(s, z) :=  u_{FS}(x, y), \qquad \bar{v}(s, z) := - \int_0^z \p_s \bar{u}, \qquad \bar{\psi}(s, z) := \int_0^z \bar{u} \\
&\lambda(s) := \Lambda(x), \qquad \lambda_G(s) := \Lambda_G(x), \qquad \mu(s) :=  \Xi(x).
\end{align}
\end{subequations}
We also define the background profiles: 
\begin{align}
[\bar{u}, \bar{v}](s, z) := [u_{FS}, v_{FS}](x, y), \qquad \bar{w}(s, z) := \bar{u}(s, z) + \eps u(s, z).  
\end{align}
In this $(s, z)$ coordinate system, the system \eqref{sideLH} reads as follows: 
\begin{align} \label{ukstruc1}
& \bar{w} \p_s u_k - \p_z \bar{w} \p_s \psi_k + \mathcal{A}[\overline{\psi}, \psi_k] - \p_z^2 u_k =  \lambda^2 \mathcal{R}_k, \qquad (s, z) \in (1, 1 + \bar{L}) \times (0, \infty). 
\end{align}
The parameter $\overline{L}$ is related to $L$ itself by $\bar{L} := \int_{1}^{1+L} \frac{1}{\Lambda(x)^2} \ud x$. The operator $\mathcal{A}$ appearing above is defined via 
\begin{align} \label{definitionA}
\mathcal{A}[\psi_k] := & \bold{c}_1 \p_z u_k + \bold{c}_2 u_k   + \bold{c}_3 \psi_k, 
\end{align}
with coefficients 
\begin{subequations}
\begin{align}
\bold{c}_1 := & \bold{C}_1(x, y) - \frac{\lambda'}{\lambda} \bar{w} z - 1_{k = 0} \eps \frac{\lambda'}{\lambda} \psi\\
\bold{c}_2 := & \bold{C}_2(x, y) + \frac{\lambda'}{\lambda} z \bar{w}_z \\
\bold{c}_3 := & \bold{C}_3(x, y) - \frac{\lambda'}{\lambda}\bar{w}_z 1_{k \ge 1} - \frac{\lambda'}{\lambda} \bar{u}_z 1_{k = 0}
\end{align}
\end{subequations}
where $\bold{C}_i$ are defined above in \eqref{boldBigC}. Associated to the transport structure of \eqref{ukstruc1}, we define 
\begin{align}
U_k := \bar{w} u_k - \bar{w}_z \psi_k.
\end{align}

\subsection{The $(s, Y)$ Coordinates}

We will next need the following variables, which are localized near the interface $\{z = 1\}$:
\begin{align} \label{change:1}
Y := \bar{w}(s, z), \qquad \phi_k(s, Y) := \psi_k(s, z), \qquad V_k(s, Y) := \p_Y^2 \phi_k(s, Y). 
\end{align}
\begin{align} \label{real:sys:omY}
Y \p_s V_k - |\bar{w}_z|^2\p_Y^2 V_k = &   \underline{\tau_1} \p_Y V_k + \underline{\tau_0} V_k + \underline{\tau_{-1}}  \p_Y \phi_k + \underline{\tau_{-2}} \phi_k +  \frac{\p_z}{\bar{w}_z} \Big( \frac{\mathcal{R}_k}{\bar{w}_z} \Big).
\end{align}
The coefficients $\underline{\tau}_i$ are defined below: 
\begin{subequations} \label{utau:def}
\begin{align}
\underline{\tau_1}(s, Y) := &\bar{w} \bar{w}_s + \bold{c}_1 \bar{w}_z - \bar{w}_{zz}, \\ 
\underline{\tau_0}(s, Y) := & \frac{\p_z}{\bar{w}_z}( \bar{w} \bar{w}_s + \bold{c}_1 \bar{w}_z - \bar{w}_{zz} ) + \frac{1}{\bar{w}_z} ( \bar{w} \bar{w}_{sz} - \bar{w}_s \bar{w}_z + \bold{c}_2 \bar{w}_z + \bold{c}_1 \bar{w}_{zz} - \bar{w}_{zzz} )   \\
\underline{\tau_{-1}}(s, Y) := & \frac{\p_z}{\bar{w}_z}\Big( \frac{ \bar{w} \bar{w}_{sz} - \bar{w}_s \bar{w}_z - \bar{w}_{zzz} + \bold{c}_2 \bar{w}_z + \bold{c}_1 \bar{w}_{zz} }{\bar{w}_z} \Big) + \frac{\bold{c}_3}{\bar{w}_z}, \\
\underline{\tau_{-2}}(s, Y) := & \frac{\p_z}{\bar{w}_z} \Big( \frac{\bold{c}_3}{\bar{w}_z} \Big).
\end{align}
\end{subequations}

\subsection{The $(s, Z)$ Coordinates}

We finally introduce the rescaling 
\begin{align} \label{Eikonal:1}
Z = \frac{Y}{\bar{w}_z(s, 0)^{\frac23}}, \qquad \Omega_k(s, Z) = V_k(s, Y),
\end{align}
where the new vorticity $\Omega_k(s, Z)$ satisfies 
\begin{align} \label{real:sys:omega}
&Z \p_s \Omega_k - \p_Z^2 \Omega_k = \tau_2 \p_Z^2 \Omega_k +   \tau_1 \p_Z \Omega_k + \tau_0 \Omega_k + \tau_{-1} \p_Z \Phi_k + \tau_{-2} \Phi_k + \bold{F}_k, 
\end{align}
where the coefficients are defined as follows:
\begin{subequations} \label{tau:def}
\begin{align}
\tau_2(s, Z) := &\frac{\bar{w}_z^2}{\bar{w}_z(s, 0)^2} \\
\tau_1(s, Z) := & - \frac23 \frac{\bar{w}_z^2 \bar{w}_{sz}(s, 0)}{\bar{w}_z(s, 0)^{\frac{11}{3}}}Z + \frac{\underline{\tau}_1}{\bar{w}_z(s, 0)^{\frac43}}, \\ 
\tau_0(s, Z) := & \frac{ \underline{\tau}_0}{\bar{w}_z(s, 0)^{\frac23}} \\ 
\tau_{-1}(s, Z) := & \frac{ \underline{\tau}_{-1}}{\bar{w}_z(s, 0)^{\frac43}} \\
\tau_{-2}(s, Z) := &   \frac{ \underline{\tau}_{-2}}{\bar{w}_z(s, 0)^{\frac23}}
\end{align}
\end{subequations}
and the forcing is defined via 
\begin{align} \label{boldF}
\bold{F}_k(s, Z) := & \frac{\p_z}{\bar{w}_z} \Big( \frac{\mathcal{R}_k}{\bar{w}_z} \Big)(s, z).
\end{align}

We introduce now the cut-off function $\chi_I$:
\begin{align} \label{chidef}
\chi_I(Z) := \chi \Big( \frac{Z}{\delta} \Big), \qquad \chi(p) := \begin{cases} 0, \qquad p < -1 \\ 1, \qquad - \frac{9}{10}< p < \frac{9}{10} \\ 0, \qquad p > 1 \end{cases}
\end{align}
where $0 < \delta << 1$ is a universal small constant. 

Subsequently, we introduce the localized vorticity as follows:
\begin{align} \label{defn:Chi:I}
\Omega_{k,I} (s, Z) := \Omega_k(s, Z) \chi_I(Z). 
\end{align}
This localized vorticity satisfies the equation 
\begin{align} \label{loc:sys:1}
&Z \p_s \Omega_{k,I} - \p_Z^2 \Omega_{k,I} =   F_{k,I}, \qquad (s, Z) \in (1, 1 + \bar{L}) \times \mathbb{R} 
\end{align}
where the forcing, $F_{k,I}$, appearing above is defined by 
\begin{align}\n
F_{k,I} := & \chi\Big(\frac{Z}{\delta}\Big) (   \tau_1 \p_Z \Omega_k + \tau_0 \Omega_k + \tau_{-1} \p_Z \Phi_k + \tau_{-2} \Phi_k)   - \frac{\chi''(\frac{Z}{\delta})}{\delta^2} \Omega_k - \frac{2}{\delta} \chi'(\frac{Z}{\delta})\p_Z \Omega_k \\  \label{loc:sys:I1}
&  + \bold{F}_k \chi\Big(\frac{Z}{\delta}\Big).
\end{align}

\subsection{The $(t, \rho)$ Coordinates}

We define now the stretched variables 
\begin{align}
t = \frac{s-1}{\overline{L}}, \qquad \rho := \frac{Z}{\overline{L}^{\frac13}}, \qquad \Xi_k(t, \rho) = \Omega_k(s, Z) \chi_I(Z). 
\end{align}
In this rescaling, we have the following system for $0 \le k \le k_{\ast}$:
\begin{subequations} \label{loc:sys:1:o:xi} 
\begin{align}
&\rho \p_t \Xi_k - \p_\rho^2 \Xi_k = \overline{L}^{\frac23} G_k, \qquad (t, \rho) \in (0, 1) \times \mathbb{R} \\
&\Xi(t, \pm \infty) = 0,  \qquad t \in (0, 1) \\
&\Xi_k(0, \rho) := \Xi^{(k)}_{Left}(\rho), \qquad 0 < \rho < \infty \\
& \Xi_k(1, \rho) := \Xi^{(k)}_{Right}(\rho), \qquad - \infty < \rho < 0.
\end{align}
\end{subequations}
We write the expression for $G_k$ as follows 
\begin{align} \n
G_k(t, \rho) = & \Big( \tau_2 \overline{L}^{-\frac13} \rho \p_\rho^2 \Xi_{k} +  \tau_1 \overline{L}^{-\frac13} \p_\rho \Xi_k  + \tau_0 \Xi_k + \tau_{-1} \p_z \psi_k \chi_I + \tau_{-2} \psi_k \chi_I \Big) \\ \label{exp:Zisit}
 & - \frac{ \chi''( \frac{\overline{L}^{\frac13}\rho}{\delta} )}{\delta^2} \Omega_k - \frac{2}{\delta} \chi'\Big( \frac{\overline{L}^{\frac13}\rho}{\delta} \Big) \p_Z \Omega_k +\bold{F}_k \chi_I.
\end{align}
The operator which takes $\Omega_{k}$ to $\Omega_{k+1}$, while scales like $\p_t$, is formally given in the following manner:
\begin{align} \label{slashD}
\Xi_{k+1} = \p_t \Xi_k +\Big( \frac{\p_t \bar{w}_z}{\bar{w}_z} \Xi_k+ \frac{1}{\bar{w}_z^2} \p_t \Big( \frac{\bar{w}_{zz}}{\bar{w}_z} \Big) u_k \Big) \chi\Big( \frac{\overline{L}^{\frac13}\rho}{\delta} \Big) = \slashed{D} \Xi_k,
\end{align}  
where we will have a use for the following ``error" term:
\begin{align} \label{defslashE}
\slashed{E}_k :=\Big(  \frac{\p_t \bar{w}_z}{\bar{w}_z} \Xi_k+ \frac{1}{\bar{w}_z^2} \p_t \Big( \frac{\bar{w}_{zz}}{\bar{w}_z} \Big) u_k\Big) \chi\Big( \frac{\overline{L}^{\frac13}\rho}{\delta} \Big).
\end{align}

\subsection{Data on the Sides}

Our main objective in this paper is to study the given data, $(F_{Left}, F_{Right})$ (and to track the required constraints on these data). We introduce the linearized good unknown: 
\begin{align}
\omega_g := \p_y u - \frac{u_{FS}''}{u_{FS}'} u,
\end{align}
which is the linearized version of $\Omega$. We will assume we are prescribed  
\begin{align}
\omega_g|_{x = 1}(y) = &F_{Left}\Big( \frac{y - y_1(\eps)}{L^{\frac13}} \Big), \qquad y > y_1(\eps)  \\
\omega_g|_{x = 1 + L}(y) = & F_{Right}\Big( \frac{y - y_2}{L^{\frac13}} \Big), \qquad y < y_2
\end{align}
We define the following quantities on the left and the right: 
\begin{align}
U_k(1, z) = & \zeta^{(k)}_{Left}(z), \qquad z > 1, \qquad U_k(1 + \bar{L}, z) =  \zeta^{(k)}_{Right}(z), \qquad 0 < z < 1, \\
\psi_k(1, z) = & \psi^{(k)}_{Left}(z), \qquad z > 1, \qquad \psi_k(1 + \bar{L}, z) = \psi^{(k)}_{Right}(z), \qquad 0 < z < 1\\
u_k(1, z) = & u^{(k)}_{Left}(z), \qquad z > 1, \qquad u_k(1 + \bar{L}, z) = u^{(k)}_{Right}(z), \qquad 0 < z < 1.
\end{align}
The main quantity we will keep track of is the following
\begin{align}
\Xi^{(k)}_{Left}(\rho) := & \Omega_{k,I}(1, Z), \qquad  \Xi^{(k)}_{Right}(\rho) :=    \Omega_{k,I}(1 + \overline{L}, Z).
\end{align}

\subsection{Function Spaces}

The following norms will be used to measure our data:
\begin{align} \label{Anorm}
\| (f, g) \|_{A} := & \| f \|_{W^{3,\infty}_{\rho}(\mathbb{R}_+)} +\| g \|_{W^{3,\infty}_{\rho}(\mathbb{R}_+)}  \\ \label{Bnorm}
\| (f, g) \|_B := & \sum_{j = 0}^4 (\| \p_z^j f \chi_{O,j} \langle z \rangle^{2M_{Max}} \|_{L^2_z(\mathbb{R}_+)} + \| \p_z^j g \chi_{O,j}  \|_{L^2_z(0,1)} ), \\ \label{Cnorm}
\| (f, g) \|_C := & \sum_{\ell = 0}^4 ( \| Z^{\ell} \p_Z^{\ell} f \|_{L^2_Z} + \| Z^{\ell} \p_Z^{\ell} g \|_{L^2_Z})
\end{align}

For the purposes of this paper, we do not need to introduce our complete $Z_k$ norm in full precision. Instead, it suffices to adopt the following bound: 
\begin{align} \n
\| (\psi, \Omega_I) \|_{Z_k} \gtrsim & \sup_s ( \| \Omega_{k,I} \|_{L^2_Z} )+ \| Z \p_Z \Omega_{k,I} \|_{L^2_Z} )  + \| \Omega_{k,I} \|_{L^\infty_Z L^6_s} + (\| u \|_{L^\infty_s L^\infty_z} + \| \psi \|_{L^\infty_s L^\infty_z}) \\ \label{calZ}
& + \sup_s \| \chi_O \p_Z^2 \Omega_{k} \langle z \rangle^{M_{Max} - k} \|_{L^2_z},
\end{align}
which is a consequence of the definition of the norm $Z_k$ from \cite{IM22}. Moreover, we prove in \cite{IM22} the bound 
\begin{align} \label{usec1}
\| (\psi, \Omega_I) \|_{Z_k} \lesssim o_L(1) \| F_{Left}, F_{Right} \|_{\mathcal{P}}.
\end{align}

\section{Derivation of Constraints}

\subsubsection*{Constraints under the $\slashed{D}$ operator}

The first thing we do is to introduce precisely the operator $\slashed{D}$ into the constraint equations. Consider the system \eqref{normdom1}. We define the modified functionals (for $k = 0, \dots, k_\ast - 1$):
\begin{align} \n
\beta^{(k)}_0 := & L^{\frac23} \int_0^\infty \Phi^0_L G^{(k)}_L -  L^{\frac23}\int_{-\infty}^0 \Phi^0_R G^{(k)}_R + \int_0^\infty \Phi^0_L \rho \slashed{E}_{k,L} - \int_{-\infty}^0 \Phi^0_R \rho \slashed{E}_{k,R} \\
\beta^{(k)}_1:= & L^{\frac23} \int_0^\infty \Phi^1_L G^{(k)}_L - L^{\frac23} \int_{-\infty}^0 \Phi^1_R G^{(k)}_R+ \int_0^\infty \Phi^1_L \rho \slashed{E}_{k,L} - \int_{-\infty}^0 \Phi^1_R \rho \slashed{E}_{k,R}  \\ \label{alphak0}
\alpha^{(k)}_0 := &-  L^{\frac23}\int G_{k+1} \Phi^0 \ud t \ud \rho + \int_0^1  \slashed{E}_k(t, 0) \ud t, \\ \label{alphak1}
\alpha^{(k)}_1 := &-  L^{\frac23}\int G_{k+1} \Phi^1  \ud t \ud \rho+  \int_0^1 \p_\rho  \slashed{E}_k(t, 0) \ud t
\end{align}

\begin{lemma} Assume the following constraints are satisfied
\begin{align} \label{dg:mod:a}
\bar{\ell}_2[\Xi^{(k)}_L, \Xi^{(k)}_R] = &  \bar{\beta}^{(k)}_2 \\  \label{dg:mod:b}
\ell_2[\Xi^{(k)}_L, \Xi^{(k)}_R] = & \beta_2^{(k)} \\  \label{dg:mod:c}
\ell_1[\Xi^{(k)}_L, \Xi^{(k)}_R] = & \beta_1^{(k)} + \alpha_1^{(k)},  \\  \label{dg:mod:d}
\ell_0[\Xi^{(k)}_L, \Xi^{(k)}_R] = & \beta_0^{(k)} + \alpha_0^{(k)}.
\end{align}
Then the unique weak solution $\Xi_k$ is $H^1_t$ and $\slashed{D} \Xi_k = \Xi_{k+1}$.
\end{lemma}
\begin{proof} This follows very closely the argument of Proposition 2.7 in \cite{DMR}. Consider $w$ to be the candidate 
\begin{subequations} \label{normdom1cand} 
\begin{align}
&\rho \p_t w - \p_\rho^2 w =  L^{\frac23} G_{k+1}, \qquad (t, \rho) \in (0, 1) \times \mathbb{R} \\
&w(t, \pm \infty) = 0,  \qquad t \in (0, 1), \\
&w_{Left}(\rho) := \Xi^{(k+1)}_L(\rho), \qquad \rho > 0, \\
&w_{Right}(\rho) := \Xi_R^{(k+1)}(\rho), \qquad \rho < 0,
\end{align}
\end{subequations}
We want to now invert $\slashed{D}$ to go from $w$ to $\Xi_k$. Recalling the expression \eqref{slashD}, we have 
\begin{align}
\Xi_k = & \Xi^{(k)}_L(\rho) + \int_0^t \p_t \Xi_k \ud t' =   \Xi^{(k)}_L(\rho) + \int_0^t (\Xi_{k+1} - \slashed{E}_k) \ud t', \qquad \rho > 0 \\
\Xi_k = & \Xi^{(k)}_R(\rho) + \int_1^t \p_t \Xi_k \ud t' =   \Xi^{(k)}_R(\rho) + \int_1^t (\Xi_{k+1} - \slashed{E}_k) \ud t', \qquad \rho < 0
\end{align}
Evaluation at $\rho = 0$ leads to the conditions 
\begin{align}
\int_0^1 \Xi_{k+1}(t, 0) \ud t = & \Xi^{(k)}_R(0) - \Xi^{(k)}_L(0) + \int_0^1 \slashed{E}_k(t, 0) \ud t \\
\int_0^1 \p_\rho \Xi_{k+1}(t, 0) \ud t = & \p_\rho \Xi^{(k)}_R(0) - \p_\rho \Xi^{(k)}_L(0) + \int_0^1 \p_\rho \slashed{E}_k(t, 0) \ud t 
\end{align} 
We now use the ``dual" identities, (2.27) in \cite{DMR} to write (for $j = 0, 1$)
\begin{align}
\int_0^1 \p_\rho^j  \Xi_{k+1}(t, 0) \ud t = \int L^{\frac23} G_{k+1} \Phi^j + \int_0^\infty \rho \Xi^{(k+1)}_L \Phi^j \ud \rho - \int_{-\infty}^0 \rho \Xi^{(k+1)}_{R} \Phi^j \ud \rho.
\end{align}
Upon equating both sides, we obtain 
\begin{align}
\p_\rho^j \Xi^{(k)}_L(0) - \p_\rho^j \Xi^{(k)}_R(0) + \int_0^\infty \rho \Xi^{(k+1)}_L \Phi^j - \int_{-\infty}^0 \rho \Xi^{(k+1)}_R \Phi^j = - \int L^{\frac23} G_{k+1} \Phi^j + \int_0^1 \p_\rho^j \slashed{E}_k(t, 0) \ud t.
\end{align}
The final step is to now use the inductive relations for $\Xi^{(k+1)}_L, \Xi^{(k+1)}_R$
\begin{align} \label{iteration:1}
\Xi^{(k+1)}_L = & \frac{\p_\rho^2 \Xi^{(k)}_L + L^{\frac23} G^{(k)}_L}{\rho} + \slashed{E}_{k,L}, \qquad \Xi^{(k+1)}_R =  \frac{\p_\rho^2 \Xi^{(k)}_R + L^{\frac23}G^{(k)}_R}{\rho} + \slashed{E}_{k,R},
\end{align}
where $\slashed{E}_{k,L} = \slashed{E}_k|_{t = 0}$ and similarly for $\slashed{E}_{k,R}$. From here, the calculation proceeds exactly as in Proposition \ref{Propcon}.
\end{proof}

\subsubsection*{Constraints on $F_{Left}, F_{Right}$}

Our next task is to convert the constraints into constraints on $F_{Left}, F_{Right}$ (which we abbreviate as $(F_L, F_R)$ for now). We briefly recall the definition of $(F_L, F_R)$ versus $(\Xi_L, \Xi_R)$: 
\begin{align}
F_L(\eta) = ( \p_y u - \frac{\p_y^2 \bar{u}_{FS}}{\p_y \bar{u}_{FS}} u)|_{x = 0}(y), \qquad \Xi_L(\rho) = ( \p_y u - \frac{\p_y^2 \bar{w}}{\p_y \bar{w}} u)|_{x = 0}(y),
\end{align}
where the fast coordinates appearing above are defined by
\begin{align}
\eta := \frac{y-y_1(\eps)}{L^{\frac13}}, \qquad \rho = \frac{Z}{L^{\frac13}}.
\end{align} 
Therefore, we need to account for two aspects when going from $F_L$ to $\Xi_L$: the change of function as well as the change of coordinate. We define the map
\begin{align} \label{IplusT}
\bold{I} + \bold{T} := (F_{L}, F_R) \mapsto (\Xi_{L}, \Xi_R),
\end{align}
where $\bold{I}$ is the identity. This map is not explicit, but is well defined. Indeed, we first define the coordinate shift (say at $t = 0$) via 
\begin{align}
\rho = \rho[F_{L}](y) = \bar{u}(1, y) + \eps u[F_{L}, \psi](1, y).
\end{align}
We are therefore able to think of $\rho = \rho(\eta)$. Above, the function $u$ is given by the integral formula 
\begin{align}
u_{L} := u[F_{L}](1, y) = \gamma_0[\psi](0) \bar{u}_z + \bar{u}_z \int_1^z \frac{F_{L}(\eta(z'))}{\bar{u}_z} \ud z'
\end{align}
Finally, we define $\bold{I} + \bold{T}$ through the map 
\begin{align} \label{IplusT}
\Xi_{L}(\rho[F_{L}](\eta)) = F_{L}(\eta) - \eps \Big( \frac{\bar{u}_z \p_z^2 u_{L} - \bar{u}_{zz} \p_z u_{L}}{\bar{u}_z (\bar{u}_z + \p_z u_{L})} \Big) u_{L}
\end{align}
We note that prescribing $(F_{L}, F_R) \in \mathcal{P}$ and then determining $(\Xi_{L}, \Xi_R)$ through the nonlinear map $\bold{I} + \bold{T}$ ensures that $(\Xi_{L}, \Xi_R)$ is in the range of $\bold{I} + \bold{T}$. This property is important to preserve, as it will ensure that we can invert the map from $(\Xi_{L}, \Xi_R)$ to $(F_{L}, F_R)$.

\subsubsection*{Intrinsic Constraints}

The first thing we do is to choose $(Q_L, Q_R)$ according to standard parabolic regularity constraints. We will choose 
\begin{align}
F_L = \widehat{F}_L + Q_L, \qquad F_R = \widehat{F}_R + Q_R, 
\end{align}
where $(Q_L, Q_R)$ will be chosen as follows. Writing the constraints for $k = 1, 2$, we find 
\begin{align}
&(k = 1) \qquad \Xi_L''(0) = - L^{\frac23} G_L(0),\\
&(k = 2) \qquad \p_{\rho}^5 \Xi_L(0) = C_1 \p_{\rho}^3 G_L(0) - L^{\frac23} G_L^{(1)}(0) - \p_{\rho}^2 \slashed{E}_{0,L}(0)
\end{align}
and symmetrically for $\Xi_R$. We then re-cast this into the corresponding constraint on $Q_L, Q_R$ by invoking \eqref{IplusT} in the following manner: 
\begin{align}
&(k = 1) \qquad Q_L''(0) = - L^{\frac23} G_L(0) - T[F_L]''(0),\\
&(k = 2) \qquad Q_L'''''(0) = C_1 \p_{\rho}^3 G_L(0) - L^{\frac23} G_L^{(1)}(0) - \p_{\rho}^2 \slashed{E}_{0,L}(0) - T[F_L]'''''(0)
\end{align}
This process can be generalized for $k$ up to $k_\ast$, To do so, we want to use the inductive relations to rewrite these as functionals acting on $F_L,F_R$ itself. It is convenient to now introduce the operators 
\begin{align}
\mathcal{D}_3 f(\rho) := \frac{\p_\rho^2 f(\rho) - \chi(\rho) \p_\rho^2 f(0)}{\rho}, \qquad \mathcal{D}_1 f(\rho) := \frac{f(\rho) - \chi(\rho) f(0)}{\rho}.
\end{align}
We next observe that inducting on \eqref{iteration:1} produces the expressions 
\begin{align} \label{exp:Xi:k:L}
\Xi^{(k)}_L = & \mathcal{D}_3^k \widehat{F}_L + L^{\frac23} \sum_{k' = 1}^k \mathcal{D}_3^{k'-1}  \mathcal{D}_1 G_L^{(k-k')} + \sum_{k' = 1}^k \mathcal{D}_3^{k'-1} \slashed{E}_{k-k',L} +  \mathcal{D}_3^{k} Q_L +  \mathcal{D}_3^{k} \bold{T}[F_L] \\ \label{exp:Xi:k:R}
\Xi^{(k)}_R = & \mathcal{D}_3^k \widehat{F}_R + L^{\frac23} \sum_{k' = 1}^k \mathcal{D}_3^{k'-1} \mathcal{D}_1 G_R^{(k-k')} + \sum_{k' = 1}^k \mathcal{D}_3^{k'-1} \slashed{E}_{k-k',R} + \mathcal{D}_3^{k} Q_R +  \mathcal{D}_3^{k} \bold{T}[F_R]
\end{align}
We now want to rewrite ($1 \le k \le k_{\ast}$)
\begin{align}
(\Xi^{(k-1)}_L)''(0) = &- L^{\frac23} G^{(k-1)}_L(0) \\
(\Xi^{(k-1)}_R)''(0) = & - L^{\frac23} G^{(k-1)}_R(0)
\end{align}
which, upon evaluating \eqref{exp:Xi:k:L} -- \eqref{exp:Xi:k:R} at $k-1$ gives the following condition on $(Q_L, Q_R)$:
\begin{align} \n
(\mathcal{D}_3^{k-1} Q_L)''(0) = &-L^{\frac23} G^{(k-1)}_L(0) - L^{\frac23} \sum_{k' = 1}^{k-1} ( \mathcal{D}_3^{k'-1}  \mathcal{D}_1 G_L^{(k-1-k')})''(0) \\
&- \sum_{k' = 1}^{k-1} ( \mathcal{D}_3^{k'-1} \slashed{E}_{k-1-k',L})''(0) -  \p_{\rho}^2 \mathcal{D}_3^{k-1} \bold{T}[F_L](0) := \bar{\beta}_2^{(k)}\\ \n
(\mathcal{D}_3^{k-1} Q_R)''(0) = &- L^{\frac23} G^{(k-1)}_R(0)- L^{\frac23} \sum_{k' = 1}^{k-1} ( \mathcal{D}_3^{k'-1}  \mathcal{D}_1 G_R^{(k-1-k')})''(0) \\
&- \sum_{k' = 1}^{k-1} ( \mathcal{D}_3^{k'-1} \slashed{E}_{k-1-k',R})''(0) -  \p_{\rho}^2 \mathcal{D}_3^{k-1} \bold{T}[F_R](0) :=  \beta_2^{(k)}
\end{align}

We now insert \eqref{exp:Xi:k:L} -- \eqref{exp:Xi:k:R} into \eqref{dg:mod:a} -- \eqref{dg:mod:d} to derive intrinsic constraints on $\Xi_L, \Xi_R$. We also need to re-define the functionals appearing on the right-hand sides of \eqref{dg:mod:a} -- \eqref{dg:mod:d}. This yields 
\begin{align}  \label{dg:m2:c}
\ell_1[\mathcal{D}_3^{k-1} \widehat{F}_L, \mathcal{D}_3^{k-1} \widehat{F}_R] = & \beta_1^{(k)} + \alpha_1^{(k)},  \\  \label{dg:m2:d}
\ell_0[\mathcal{D}_3^{k-1} \widehat{F}_L, \mathcal{D}_3^{k-1} \widehat{F}_R] = & \beta_0^{(k)} + \alpha_0^{(k)}, \\
\p_{\rho}^{2 + 3(k-1)} Q_L(0) = & \bar{ \beta}_2^{(k)}, \\
\p_{\rho}^{2 + 3(k-1)} Q_R(0) = &\beta_2^{(k)}
\end{align}
where we redefine the right-hand side functionals as follows:
\begin{align} \n
\beta^{(k)}_0 := & L^{\frac23} \int_0^\infty \Phi^0_L G^{(k)}_L -  L^{\frac23}\int_{-\infty}^0 \Phi^0_R G^{(k)}_R + \int_0^\infty \Phi^0_L \rho \slashed{E}_{k,L} - \int_{-\infty}^0 \Phi^0_R \rho \slashed{E}_{k,R} \\  \n
& - L^{\frac23} \sum_{k'=1}^k \ell_0[ \mathcal{D}_3^{k'-1} \mathcal{D}_1 G_L^{(k-k')}, \mathcal{D}_3^{k'-1} \mathcal{D}_1 G_R^{(k-k')}]  \\ \n
&- \sum_{k'=1}^k \ell_0[ \mathcal{D}_3^{k'-1}  \slashed{E}_L^{(k-k')}, \mathcal{D}_3^{k'-1}  \slashed{E}_R^{(k-k')}], \\  \label{beta:1:1:1}
& - \sum_{k'=1}^k \ell_0[ \mathcal{D}_3^{k'}  Q_L, \mathcal{D}_3^{k'} Q_R] - \sum_{k'=1}^k \ell_0[ \mathcal{D}_3^{k'}   \bold{T}[F_L], \mathcal{D}_3^{k'} \bold{T}[F_R]] \\ \n
\beta^{(k)}_1:= & L^{\frac23} \int_0^\infty \Phi^1_L G^{(k)}_L - L^{\frac23} \int_{-\infty}^0 \Phi^1_R G^{(k)}_R+ \int_0^\infty \Phi^1_L \rho \slashed{E}_{k,L} - \int_{-\infty}^0 \Phi^1_R \rho \slashed{E}_{k,R}  \\ \n
&  - L^{\frac23} \sum_{k'=1}^k \ell_1[ \mathcal{D}_3^{k'-1} \mathcal{D}_1 G_L^{(k-k')}, \mathcal{D}_3^{k'-1} \mathcal{D}_1 G_R^{(k-k')}]  \\ \n 
&- \sum_{k'=1}^k \ell_1[ \mathcal{D}_3^{k'-1}  \slashed{E}_L^{(k-k')}, \mathcal{D}_3^{k'-1}  \slashed{E}_R^{(k-k')}] \\ \label{beta:1:1:2}
& - \sum_{k'=1}^k \ell_1[ \mathcal{D}_3^{k'}  Q_L, \mathcal{D}_3^{k'} Q_R] - \sum_{k'=1}^k \ell_1[ \mathcal{D}_3^{k'}   \bold{T}[F_L], \mathcal{D}_3^{k'} \bold{T}[F_R]]
\end{align}
and $\alpha_0^{(k)}, \alpha_1^{(k)}$ are defined in \eqref{alphak0} -- \eqref{alphak1}.

\subsubsection*{Full Constraint System}

From the above developments, we are able to state our full set of constraints. To do so, we first define $\ell_{-1}$ motivated by \eqref{mc1} as follows: 
\begin{align} \label{mc1}
\ell_{-1}[F_L, F_R] :=  \int_0^1 \frac{F_{R}(L^{-\frac13}(z'-1))}{\overline{u_{FS}}'(1, z')} \ud z'
\end{align} 
We now have the following 
\begin{proposition} Let $1 \le k \le k_\ast$. Suppose 
\begin{align} \label{dg:m4:a}
\p_{\rho}^{2 + 3(k-1)}Q_L(0)= &  \bar{\beta}^{(k)}_2, \\  \label{dg:m4:b}
\p_{\rho}^{2 + 3(k-1)}Q_R(0)= & \beta_2^{(k)}, \\  \label{dg:m4:c}
\ell_1[\mathcal{D}_3^{k-1} \widehat{F}_L, \mathcal{D}_3^{k-1} \widehat{F}_R] = & \beta_1^{(k)} + \alpha_1^{(k)},  \\  \label{dg:m4:d}
\ell_0[\mathcal{D}_3^{k-1} \widehat{F}_L, \mathcal{D}_3^{k-1} \widehat{F}_R] = & \beta_0^{(k)} + \alpha_0^{(k)}, \\ \label{dg:m4:e}
\ell_{-1}[F_L, F_R] = & 0,
\end{align}
where $\alpha^{(k)}_1, \alpha^{(k)}_0$ are defined in \eqref{alphak0} -- \eqref{alphak1}, and $\beta^{(k)}_j$ are defined in \eqref{beta:1:1:1} -- \eqref{beta:1:1:2}. Then there exists a unique $H^{k_\ast}_t$ solution to the system
\begin{subequations}
\begin{align}
&\rho \p_t \Xi - \p_\rho^2 \Xi = L^{\frac23}G \\
&\Xi|_{t = 0} = \Xi_L  = (\bold{I} + \bold{T}) F_L, \qquad \rho > 0 \\
&\Xi|_{t = 1} = \Xi_R  = (\bold{I} + \bold{T}) F_R, \qquad \rho < 0 \\
&\Xi|_{|\rho| \rightarrow \infty} = 0.
\end{align}
\end{subequations}
\end{proposition}

\section{Proof of Main Result}

We define 
\begin{align}
\ell_j^{(k)}[\widehat{F}_L, \widehat{F}_R] = & \ell_j[\mathcal{D}_3^{k-1} \widehat{F}_L, \mathcal{D}_3^{k-1} \widehat{F}_R], \qquad j = 0, 1, 2, \qquad k = 1, \dots, k_\ast  \\
\bar{\ell}_2^{(k)}[\widehat{F}_L, \widehat{F}_R] =& \bar{\ell}_2[\mathcal{D}_3^{k-1} \widehat{F}_L, \mathcal{D}_3^{k-1} \widehat{F}_R], \qquad k = 1, \dots, k_\ast 
\end{align}
We regard the functionals $\beta_j^{(k)}, \bar{\beta}_2^{(k)}$ as depending on both $[F_L, F_R]$ as well as the solution $\psi$, whereas $\alpha^{(k)}_j$ only depends on the interior solution $\psi$. Therefore, we rewrite \eqref{dg:m4:a} -- \eqref{dg:m4:e} by making this dependence explicit:  
\begin{align} \label{dg:m5:a}
\p_{\rho}^{2 + 3(k-1)}Q_L(0) = &  \bar{\beta}^{(k)}_2[(F_L, F_R), \psi], \\  \label{dg:m5:b}
\p_{\rho}^{2 + 3(k-1)}Q_R(0)= & \beta_2^{(k)}[(F_L, F_R), \psi], \\  \label{dg:m5:c}
\ell_1^{(k)}[\widehat{F}_L, \widehat{ F}_R] = & \beta_1^{(k)}[(F_L, F_R), \psi] + \alpha_1^{(k)}[\psi],  \\  \label{dg:m5:d}
\ell_0^{(k)}[\widehat{F}_L, \widehat{F}_R] = & \beta_0^{(k)}[(F_L, F_R), \psi] + \alpha_0^{(k)}[\psi], \\ \label{dg:m5:e}
\ell_{-1}[\widehat{F}_L, \widehat{F}_R] = & 0,
\end{align}
We now formulate a proposition which will imply our main result by proving estimate \eqref{bbd}. 
\begin{proposition} \label{prop:real} Fix any $(\bar{F}_L, \bar{F}_R) \in \bar{\mathcal{P}}_{k_\ast}$. Let $L << 1$. Then there exist numbers 
\begin{align} \label{A:def:2}
\bold{A}_{k_\ast} := (c_{-1}, c^{(k)}_0, c^{(k)}_1, q_{L, 2 + 3(k-1)}, q_{R, 2 + 3(k-1)}) \text{ for }k = 1, \dots, k_{\ast}
\end{align}
such that data of the type 
\begin{align}
(F_L, F_R) =& (\widehat{F}_L, \widehat{F}_R) + (Q_L, Q_R) \\ \label{data:set:k:F}
(\widehat{F}_L, \widehat{F}_R) =& ( \bar{F}_L, \bar{F}_R) +c_{-1}\bold{e}_{-1} + \sum_{k = 1}^{k_\ast} (c_0^{(k)} \bold{e}^{(k)}_0 + c^{(k)}_1 \bold{e}^{(k)}_1) \\
(Q_L, Q_R) = & (\chi(\rho) \sum_{k = 1}^{k_\ast} q_{L, 2 + 3(k-1)} \rho^{2 + 3(k-1)} , \chi(\rho) \sum_{k = 1}^{k_\ast} q_{R, 2 + 3(k-1)} \rho^{2 + 3(k-1)}  )
\end{align}
satisfy the constraint equations \eqref{dg:m5:a} -- \eqref{dg:m5:e}. The scalars in $\bold{A}_{k_\ast}$ satisfy the bounds 
\begin{align}
|\bold{A}_{k_\ast}|_{\mathbb{R}^{4k_\ast + 1}} \lesssim o_L(1) \| \psi \|_{\mathcal{Z}}.
\end{align}

\end{proposition} 

The proposition is a consequence of two lemmas: one regarding the numbers $\alpha_j^{(k)}$, and one regarding the Lipschitz functions $\beta_j^{(k)}, \bar{\beta}_2^{(k)}$. 


\begin{lemma}  \label{lem:2} The numbers $\alpha^{(k)}_j[\psi]$ satisfy the bound 
\begin{align}
|\alpha^{(k)}_j[\psi]| \lesssim o_L(1)  \| \psi \|_{\mathcal{Z}}, \qquad k = 1, \dots, k_\ast . 
\end{align}
\end{lemma}

\begin{lemma}  \label{lem:3} For $j = 0, 1, 2$ and $k = 1, \dots, k_\ast $, we have 
\begin{align}
\beta_j^{(k)}[(F_L, F_R), \psi] = \gamma_j^{(k)}[F_L, F_R] + \Gamma_j^{(k)}[\psi], 
\end{align}
and the following bounds hold 
\begin{align}
| \gamma_j^{(k)}[F_L, F_R] -  \gamma_j^{(k)}[\bar{F}_L, \bar{F}_R]| \lesssim &o_L(1) \| (F_L - \bar{F}_L, F_R - \bar{F}_R) \|_{\mathcal{P}} \\
|\Gamma^{(k)}_j[\psi]| \lesssim & o_L(1) \| \psi \|_{\mathcal{Z}}.
\end{align}
\end{lemma}

A standard contraction mapping argument then proves Proposition \ref{prop:real}. The rest of the paper is devoted to establishing these three lemmas. 

\begin{proof}[Proof of Lemma \ref{lem:2}] We recall the definitions \eqref{alphak0} -- \eqref{alphak1}. First, we treat the $\slashed{E}_k$ terms. Recalling the definition \eqref{defslashE}, we have the following identity  
\begin{align} \n
\p_{\rho} \slashed{E}_{k-1} =& \overline{L}^{\frac43} \p_Z \{ \frac{\bar{w}_{sz}}{\bar{w}_z} \} \Omega_{k-1,I}+ \overline{L}^{\frac43}  \frac{\bar{w}_{sz}}{\bar{w}_z} \p_Z \Omega_{k-1,I} + \overline{L}^{\frac43} \p_Z \{ \frac{1}{\bar{w}_z^2} \p_s \Big( \frac{\bar{w}_{zz}}{\bar{w}_z} \Big) \} u_{k-1} \chi_I \\
&+ \overline{L}^{\frac43}  \frac{1}{\bar{w}_z^2} \p_s \Big( \frac{\bar{w}_{zz}}{\bar{w}_z} \Big) \p_Z u_{k-1} \chi_I 
\end{align}
after which we estimate 
\begin{align*}
\| \p_{\rho} \slashed{E}^{(n)}_{k-1}(\cdot, 0) \|_{L^1_t} \lesssim &  \overline{L}^{\frac13} (\| \Omega^{(n)}_{k-1,I} \|_{L^\infty_Z L^1_s} + \| \p_Z \Omega^{(n)}_{k-1,I} \|_{L^\infty_Z L^1_s} + \| \p_Z \phi_{k-1} \|_{L^\infty_Z L^1_s}) \lesssim o_L(1) \| \psi^{(n)} \|_{Z_{k-1}}, \\
\| \slashed{E}^{(n)}_{k-1}(\cdot, 0) \|_{L^1_t} \lesssim & \| \Omega^{(n)}_{k-1,I} \|_{L^\infty_Z L^1_s} + \| \p_Z \Phi^{(n)}_{k-1} \|_{L^\infty_Z L^1_s} \lesssim o_L(1) \| \psi^{(n)} \|_{Z_{k-1}}.
\end{align*}
We now turn to the $G_{k+1}$ terms, defined in \eqref{exp:Zisit}. Of these, the most dangerous is the following: 
\begin{align} \n
&L^{\frac23} |\int_{t} \int_{\rho} \tau_1 \Phi^{(j)} \p_Z \Omega_k^{(n)} \ud \rho \ud t| \\ \n
= &L^{\frac13} |\int_{t} \int_{\rho} \tau_1 \Phi^{(j)} \p_{\rho} \Omega_k^{(n)} \ud \rho \ud t| \\ \n
\le& L^{\frac13} |\int_{t} \int_{\rho} \tau_1 \p_\rho \Phi^{(j)}  \Omega_k^{(n)} \ud \rho \ud t| + L^{\frac23}| \int_t \int_\rho \p_Z \tau_1 \Phi^j \Omega_k \ud \rho \ud t | + L^{\frac13} |\int_t \tau_1 \Phi^j \Omega_k \ud t|   \\ \n
\lesssim & L^{\frac13} \| \langle \rho \rangle \p_\rho \Phi^{(j)} \|_{L^2_{t \rho}} \| \Omega_k \|_{L^\infty_Z L^2_t} \\ \n
\lesssim & L^{\frac13} \| \langle \rho \rangle \p_\rho \Phi^{(j)} \|_{L^2_{t \rho}} \| \Omega_k \|_{L^\infty_Z L^6_t} \\ \n
\lesssim & L^{\frac16} \| \langle \rho \rangle \p_\rho \Phi^{(j)} \|_{L^2_{t \rho}} \| \Omega_k \|_{L^\infty_Z L^6_s} \\ \label{sixterm}
\lesssim & L^{\frac16} \| \psi \|_{\mathcal{Z}},  
\end{align}
according to the lower bound \eqref{calZ}. Similarly, we have the term 
\begin{align*}
&L^{\frac23} | \int_t \int_{\rho} \Phi^{(j)} \tau_2 L^{-\frac13} \rho \p_{\rho}^2 \Xi_k \ud t \ud \rho | \\
\lesssim &L^{\frac23} | \int_t \int_{\rho} \p_{\rho} \Phi^{(j)} \tau_2 L^{-\frac13} \rho \p_{\rho} \Xi_k \ud t \ud \rho | + L^{\frac23} | \int_t \int_{\rho} \Phi^{(j)} \tau_2 L^{-\frac13} \p_{\rho} \Xi_k \ud t \ud \rho |  \\
& + L^{\frac23} | \int_t \int_{\rho}  \Phi^{(j)} \p_{\rho}\tau_2 L^{-\frac13} \rho \p_{\rho} \Xi_k \ud t \ud \rho | = I_1 + I_2 + I_3.
\end{align*}
While $I_2, I_3$ are treated identically to \eqref{sixterm}, we estimate $I_1$ as follows: 
\begin{align*}
I_1 \lesssim &L^{\frac13} \| \langle \rho \rangle \p_\rho \Phi^{(j)} \|_{L^2_{t \rho}} \| \rho \p_{\rho} \Xi_k \|_{L^2_{t\rho}} \\
\lesssim& L^{\frac13} \| \langle \rho \rangle \p_\rho \Phi^{(j)} \|_{L^2_{t \rho}} \| Z \p_{Z} \Omega_k \|_{L^2_{t\rho}}  \\
\lesssim& L^{\frac13} \| \langle \rho \rangle \p_\rho \Phi^{(j)} \|_{L^2_{t \rho}} \| Z \p_{Z} \Omega_k \|_{L^\infty_{t} L^2_{\rho}} \\
 \lesssim& L^{\frac16} \| \langle \rho \rangle \p_\rho \Phi^{(j)} \|_{L^2_{t \rho}} \| Z \p_{Z} \Omega_k \|_{L^\infty_{s} L^2_{Z}} \\
 \lesssim & L^{\frac16} \| \psi \|_{\mathcal{Z}}, 
\end{align*}
again according to the lower bound, \eqref{calZ}. The remaining terms in $G_{k+1}$ are strictly simpler, and we omit repeating these details. 
\end{proof}

\begin{proof}[Proof of Lemma \ref{lem:3}] We first recall the following bounds, which were established in \cite{IM22}. There is a decomposition 
\begin{align}
G^{(k)}_{L} =& G^{(k, d)}_{L}[F_L, F_R] + G^{(k,b)}_{L}[F_L, F_R] + G^{(k,s)}_{L}[\psi], \\
 G^{(k)}_{R} =& G^{(k, d)}_{R}[F_L, F_R] + G^{(k,b)}_{R}[F_L, F_R] + G^{(k,s)}_{R}[\psi]
\end{align}
where the following bounds hold 
\begin{align}
\| \langle \rho \rangle^{M_m} \p_\rho^m G^{(k,d)}_{L,R}[F_L, F_R] \|_\infty \lesssim & L^{- \frac13} \| (F_L, F_R) \|_{\mathcal{P}} \\
\|  \langle z \rangle^{M_m} \p_z^m G^{(k, b)}_{L,R}[F_L, F_R] \|_{L^\infty} \lesssim & L^{\frac16} \| (F_L, F_R) \|_{\mathcal{P}}, \\
\|  \langle z \rangle^{M_m} \p_z^m G^{(k, s)}_{L,R}[\psi] \|_{L^\infty} \lesssim& o_L(1)  \sum_{k' = 0}^k \| \psi \|_{Z_k}.
\end{align}
Pairing these bounds with the definitions \eqref{beta:1:1:1} -- \eqref{beta:1:1:2} gives the desired result. 
\end{proof}

\begin{proof}[Proof of Theorem \ref{thm:twomain}] The proof essentially follows from the bounds \eqref{abd} -- \eqref{bbd}, which have now been established thanks to Proposition \ref{prop:real}. We introduce the following iteration for $0 \le k \le k_{\ast} -1$: 
\begin{subequations} \label{it:it:1} 
\begin{align}
&\rho \p_t \Xi^{(n+1)}_k - \p_\rho^2 \Xi^{(n+1)}_k = \overline{L}^{\frac23} G^{(n)}_k, \qquad (t, \rho) \in (0, 1) \times \mathbb{R} \\
&\Xi_k^{(n+1)}(t, \pm \infty) = 0,  \qquad t \in (0, 1) \\
&\Xi^{(n+1)}_k(0, \rho) := \Xi^{(k, n+1)}_{Left}(\rho), \qquad 0 < \rho < \infty \\
& \Xi^{(n+1)}_k(1, \rho) := \Xi^{(k, n+1)}_{Right}(\rho), \qquad - \infty < \rho < 0.
\end{align}
\end{subequations}
The source term $G_k^{(n)}$ is given as follows 
\begin{align} \n
G^{(n)}_k(t, \rho) = & \Big( \tau^{(n)}_2 \overline{L}^{-\frac13} \rho \p_\rho^2 \Xi^{(n)}_{k} +  \tau^{(n)}_1 \overline{L}^{-\frac13} \p_\rho \Xi^{(n)}_k  + \tau^{(n)}_0 \Xi^{(n)}_k + \tau^{(n)}_{-1} \p_z \psi^{(n)}_k \chi_I + \tau^{(n)}_{-2} \psi^{(n)}_k \chi_I \Big) \\ \label{exp:Zisit:4it}
 & - \frac{ \chi''( \frac{\overline{L}^{\frac13}\rho}{\delta} )}{\delta^2} \Omega^{(n)}_k - \frac{2}{\delta} \chi'\Big( \frac{\overline{L}^{\frac13}\rho}{\delta} \Big) \p_Z \Omega^{(n)}_k +\bold{F}^{(n)}_k \chi_I.
\end{align}
The iteration is initialized by setting 
\begin{align}
\Xi_k^{(0)} = 0, \qquad \psi_k^{(0)} = 0, \qquad  \Omega^{(0)}_k = 0 \qquad G_k^{(0)} = 0, \qquad \bold{F}^{(0)}_k = 0.  
\end{align}
We subsequently solve \eqref{it:it:1} to obtain $\Xi^{(n+1)}$. Since $G^{(n)}$ depends not only on $\Xi^{(n+1)}$, we need to define $\psi^{(n+1)}$ and $\Omega^{(n+1)}$. Over the course of the iteration, we will relax the constraint that $\Xi = \Omega \chi_I$. That is, we will not enforce $\Xi^{(n+1)} = \Omega^{(n+1)} \chi_I$. Instead, we will define $\Omega^{(n+1)}$ as a different object: it is derived from solving the stream function equation, with the data $[\psi^{(k, n+1)}_{Left}, \psi^{(k, n+1)}_{Right}]$. That is, consider the solution to the full Prandtl system \eqref{ukstruc1}
\begin{align}
& \bar{w} \p_s u - \p_z \bar{w} \p_s \psi + \mathcal{A}[\overline{\psi}, \psi] - \p_z^2 u = 0, \qquad (s, z) \in (0, 1) \times (0, 1) \\
& u|_{z = 0} = 0, \qquad \omega_g|_{z = 0} = \Xi^{(n+1)}|_{Z = 0}, \\
& \psi|_{s = 1} = \psi^{(k, n+1)}_{Left}(z), \qquad 0 < z < 1,
\end{align}
and analogously for the $1 < z < \infty$ domain. This will define $\psi^{(n+1)}$ on the tangential interior $0 < s < 1$, and is a smooth function away from $z = 1$.  Importantly, this structure ensures that $\psi^{(n+1)}$ can in turn be controlled by re-applying the energy analysis from \cite{IM22}.

We now need to define the map that takes us from $F^{(n+1)}_{Left}, F^{(n+1)}_{Right}$ to $\Xi^{(n+1)}_{Left}, \Xi^{(n+1)}_{Right}$, in a similar manner to \eqref{IplusT}. This map is not explicit, but is well defined. Indeed, we first define the coordinate shift via 
\begin{align}
\rho = \rho[F_{Left}^{(n)}](y) = \bar{u}(1, y) + \eps u[F^{(n)}_{Left}, \psi^{(n)}](1, y).
\end{align}
Above, the function $u$ is given by the integral formula 
\begin{align}
u^{(n)}_{Left} := u[F^{(n)}_{Left}](1, y) = \gamma_0^{(0, n)}[\psi^{(n)}](0) \bar{u}_z + \bar{u}_z \int_1^z \frac{F^{(n)}_{Left}(\eta(z'))}{\bar{u}_z} \ud z'
\end{align}
Finally, we define 
\begin{align}
\bold{T}^{(n)} := F^{(n+1)}_{Left} \mapsto \Xi^{(n+1)}_{Left}
\end{align}
through the map 
\begin{align}
\Xi^{(n+1)}_{Left}(\rho[F^{(n)}_{Left}](\eta)) = F^{(n+1)}_{Left}(\eta) - \eps \Big( \frac{\bar{u}_z \p_z^2 u^{(n)}_{Left} - \bar{u}_{zz} \p_z u^{(n)}_{Left}}{\bar{u}_z (\bar{u}_z + \p_z u^{(n)}_{Left})} \Big) u^{(n)}_{Left}
\end{align}
We note that prescribing $F^{(n+1)}_{Left} \in \mathcal{P}$ at each step of the iteration and then determining $\Xi^{(n+1)}_{Left}$ through the nonlinear map $\bold{T}$ ensures that $\Xi^{(n+1)}_{Left}$ is in the range of $\bold{T}$. This property is important to preserve, as it will ensure that we can invert the map from $\Xi_{Left}$ to $F_{Left}$. 

Due to the expressions \eqref{exp:Xi:k:L} -- \eqref{exp:Xi:k:R}, we regard $\Xi^{(k, n)}_{L}, \Xi^{(k, n)}_{R}$ as functionals of the given data for the $k$ prior iterates. Moreover, we subsequently view each quantity below also as functionals of these given data: 
\begin{subequations}
\begin{align} \label{ttu1}
\Xi^{(k, n)}_{Left} =& \widetilde{\Xi}^{(k)}_{Left}[ \Xi^{(0,n)}_{Left}, \Xi^{(0,n-1)}_{Left}, \dots, \Xi^{(0,n-k)}_{Left}, \psi^{(n-1)}, \dots, \psi^{(n-k)}], \\ \label{ttu2}
G^{(k, n)}_{Left} = & \widetilde{G}^{(k)}_{Left}[ \Xi^{(0,n)}_{Left}, \Xi^{(0,n-1)}_{Left}, \dots, \Xi^{(0,n-k)}_{Left}, \psi^{(n)}, \dots, \psi^{(n-k)}], \\  \label{ttu3}
\psi^{(k, n)}_{Left} = & \widetilde{\psi}^{(k)}_{Left}[\Xi^{(0,n)}_{Left}, \Xi^{(0,n-1)}_{Left}, \dots, \Xi^{(0,n-k)}_{Left}, \psi^{(n)}, \dots, \psi^{(n-k)}], \\ \label{ttu4}
\zeta^{(k, n)}_{Left} = & \widetilde{\zeta}^{(k)}_{Left}[\Xi^{(0,n)}_{Left}, \Xi^{(0,n-1)}_{Left}, \dots, \Xi^{(0,n-k)}_{Left}, \psi^{(n)}, \dots, \psi^{(n-k)}].
\end{align}
\end{subequations}

We will prescribe data of the type
\begin{align}
(F^{(n+1)}_L, F^{(n+1)}_R) =& (\widehat{F}^{(n+1)}_L, \widehat{F}^{(n+1)}_R) + (Q^{(n+1)}_L, Q^{(n+1)}_R) \\ \label{data:set:k:F}
(\widehat{F}^{(n+1)}_L, \widehat{F}^{(n+1)}_R) =& ( \bar{F}_L, \bar{F}_R) +c^{(n+1)}_{-1}\bold{e}_{-1} + \sum_{k = 1}^{k_\ast} (c_0^{(k, n+1)} \bold{e}^{(k)}_0 + c^{(k, n+1)}_1 \bold{e}^{(k)}_1) \\
(Q^{(n+1)}_L, Q^{(n+1)}_R) = & (\chi(\rho) \sum_{k = 1}^{k_\ast} q^{(n+1)}_{L, 2 + 3(k-1)} \rho^{2 + 3(k-1)} , \chi(\rho) \sum_{k = 1}^{k_\ast} q^{(n+1)}_{R, 2 + 3(k-1)} \rho^{2 + 3(k-1)}  )
\end{align}
Importantly, the basis functions $\bold{e}_0^{(k)}, \bold{e}_1^{(k)}, \bold{e}_{-1}$ are remaining fixed through the iteration; only the projections 
\begin{align*}
\bold{A}_{k_\ast}^{(n+1)} := (c^{(n+1)}_{-1}, c^{(k, n+1)}_0, c^{(k, n+1)}_1, q^{(n+1)}_{L, 2 + 3(k-1)}, q^{(n+1)}_{R, 2 + 3(k-1)}) \text{ for }k = 1, \dots, k_{\ast}
\end{align*}
are updating (due to the update in the source term $G^{(n)}$). 

 By re-applying Proposition \ref{prop:real}, we obtain the bounds 
\begin{align}
|\bold{A}_{k_\ast}^{(n+1)}|_{\mathbb{R}^{4k_\ast + 1}} \lesssim & o_L(1) \sum_{k = 0}^{k_\ast} \| \psi^{(n - k)} \|_{\mathcal{Z}} +  o_L(1) \sum_{k = 0}^{k_\ast} | \bold{A}_{k_\ast}^{(n - k)} |_{\mathbb{R}^{4k_\ast + 1}}, \\ \n
\| \psi^{(n+1)} \|_{\mathcal{Z}} \lesssim & \| F^{(n+1)}_{Left}, F^{(n+1)}_{Right} \|_{\mathcal{P}} + o_L(1) \Big( \sum_{k' = 0}^{k_{\ast}} \| F^{(n-k')}_{Left}, F^{(n-k')}_{Right} \|_{\mathcal{P}} +  \sum_{k' = 0}^{k_{\ast}} \| \psi^{(n-k')} \|_{Z_{k_{\ast}}}  ) \\
\lesssim &  \| \bar{F}_{Left}, \bar{F}_{Right} \|_{\bar{\mathcal{P}}_{k_\ast}} + o_L(1) \sum_{k = 0}^{k_\ast+1} |\bold{A}_{k_\ast}^{(n+1 - k)}|_{\mathbb{R}^{4k_\ast + 1}} +  o_L(1) \sum_{k' = 0}^{k_{\ast}} \| \psi^{(n-k')} \|_{Z_{k_{\ast}}}.
\end{align}
The dependance of the above bounds on the $k_\ast$ prior iterates (as opposed to just the previous iterate) is due to the ``lagging" structure of the interdependencies, \eqref{ttu1} -- \eqref{ttu4}. By taking differences, it is clear that the sequence $\{ \bold{A}_{k_\ast}^{(n)}, \psi^{(n)} \}_{n = 0}^\infty$ converges to $\{ \bold{A}_{k_\ast}, \psi \}_{n = 0}^\infty$, as required. 
\end{proof}

We now address the case of $L$ not necessarily small. 

\begin{proposition} \label{prop:real} Fix any $(\bar{F}_L, \bar{F}_R) \in \bar{\mathcal{P}}_{k_\ast}$. There exists a discrete set $\{L_i \}_{i = 1}^N$ such that if $L \neq L_i$, then there exist numbers 
\begin{align} \label{A:def:2}
\bold{A}_{k_\ast} := (c_{-1}, c^{(k)}_0, c^{(k)}_1, q_{L, 2 + 3(k-1)}, q_{R, 2 + 3(k-1)}) \text{ for }k = 1, \dots, k_{\ast}
\end{align}
such that data of the type 
\begin{align}
(F_L, F_R) =& (\widehat{F}_L, \widehat{F}_R) + (Q_L, Q_R) \\ \label{data:set:k:F}
(\widehat{F}_L, \widehat{F}_R) =& ( \bar{F}_L, \bar{F}_R) +c_{-1}\bold{e}_{-1} + \sum_{k = 1}^{k_\ast} (c_0^{(k)} \bold{e}^{(k)}_0 + c^{(k)}_1 \bold{e}^{(k)}_1) \\
(Q_L, Q_R) = & (\chi(\rho) \sum_{k = 1}^{k_\ast} q_{L, 2 + 3(k-1)} \rho^{2 + 3(k-1)} , \chi(\rho) \sum_{k = 1}^{k_\ast} q_{R, 2 + 3(k-1)} \rho^{2 + 3(k-1)}  )
\end{align}
satisfy the constraint equations \eqref{dg:m5:a} -- \eqref{dg:m5:e}. The scalars in $\bold{A}_{k_\ast}$ satisfy the bounds 
\begin{align}
|\bold{A}_{k_\ast}|_{\mathbb{R}^{4k_\ast + 1}} \lesssim  \| \bar{F}_L, \bar{F}_R \|_{\bar{\mathcal{P}}_{k_\ast}}.
\end{align}

\end{proposition} 
\begin{proof} We first recall from \cite{IM22} the following stability bound: 
\begin{align} \label{lab1}
\| \psi\|_{Z_{k_{\ast}}} \lesssim & \| F_{Left}, F_{Right} \|_{\mathcal{P}}.
\end{align} 
We now need to address the constraint equations. We define the vector $\bold{A}_{k_\ast}$ as before. We subsequently view the following as operators $\mathbb{R}^{4k_{\ast}+1} \mapsto \mathbb{R}$.  
\begin{align}
\alpha_j^{(k)} = \alpha_j^{(k)}(\bold{A}_{k_\ast}; L), \qquad \Gamma_{j}^{(k)} = \Gamma_{j}^{(k)}(\bold{A}_{k_\ast}, L).
\end{align}
We thus derive the following affine system for $\vec{\mathcal{A}}$:
\begin{align}
\bold{A}_{k_\ast} - \bold{M}_L \bold{A}_{k_\ast} =  g(\overline{F}_{Left},\overline{F}_{Right}),
\end{align} 
where the operator $\bold{M}_L = M_L + B_L$, and $M_L$ is the linear operator that takes $\bold{A}_{k_\ast} \mapsto \psi \mapsto \alpha_j^{(k)}(\bold{A}_{k_\ast}; L) +  \Gamma_{j}^{(k)}(\bold{A}_{k_\ast}, L)$ for $1 \le j \le 4$ and $1 \le k \le k_\ast$, and $B_L$ is the linear operator that takes $\bold{A}_{k_\ast}$ and outputs the numbers $\gamma(\cdot, L)$ in \eqref{dg:m5:a}. Through the analytic dependence of $\bold{M}_L$ on $L$, for $L \ge L_0$ (due to the rescaling $F_{Left}(\frac{Z}{L^{\frac13}}), F_{Right}(\frac{Z}{L^{\frac13}})$), we see that by excising finitely many values of $L$, we can obtain 
\begin{align} \label{Abd}
\bold{A}_{k_\ast} = (\mathbb{I} - \bold{M}_L)^{-1} g(\overline{F}_{Left}, \overline{F}_{Right}), 
\end{align}
which shows that $|\bold{A}_{k_\ast}| \lesssim \| \overline{F}_{Left}, \overline{F}_{Right} \|_{\bar{\mathcal{P}}_{k_\ast}}$. Therefore, combining \eqref{lab1} with \eqref{Abd} gives the bounds 
\begin{align}
|\bold{A}_{k_\ast}| + \| \psi\|_{Z_{k_{\ast}}} \lesssim & \| \overline{F}_{Left}, \overline{F}_{Right} \|_{\bar{\mathcal{P}}_{k_\ast}}.
\end{align} 
\end{proof}

\noindent \textbf{Acknowledgements:} S.I is grateful for the hospitality and inspiring work atmosphere at NYU Abu Dhabi and Courant Institute, NYU where part of this work took place. The work of S.I is partially supported by a UC Davis startup grant. The work of N.M. is supported by NSF grant DMS-1716466 and by Tamkeen under the NYU Abu Dhabi Research Institute grant of the center SITE. The authors would also like to thank A-L. Dalibard and F. Marbach for helpful discussions about the work \cite{DMR} and for pointing out the requirement for the construction argument implemented in this work.

\def\bibindent{3.5em}

\end{document}